\documentclass[11pt]{amsart}

\usepackage[margin=1in]{geometry}
\usepackage{amsmath,amssymb,amsthm}
\usepackage{enumitem}
\usepackage[hidelinks]{hyperref}

\usepackage{tikz}
\usepackage{tikz-cd}

\newcommand{\Z}{\mathbb{Z}}
\newcommand{\F}{\mathbb{F}}
\newcommand{\Q}{\mathbb{Q}}
\newcommand{\R}{\mathcal{R}}
\newcommand{\ab}{\mathrm{ab}}
\newcommand{\Cl}{\mathrm{Cl}}
\newcommand{\rk}{\mathrm{rk}}

\newcommand{\Hp}{\mathbf{H}_p}
\newcommand{\Gp}{\mathbf{G}_p}
\newcommand{\Ab}{\mathbf{Ab}}
\newcommand{\Abp}{\mathbf{Ab}_{p-1}}
\newcommand{\cO}{\mathcal{O}}
\newcommand{\Gal}{\mathrm{Gal}}
\newcommand{\Frob}{\mathrm{Frob}}
\newcommand{\fr}{\mathrm{fr}}
\newcommand{\pr}{\mathrm{pr}}
\newcommand{\PP}{\mathcal{P}}
\newcommand{\eps}{\epsilon}

\newtheorem{thm}{Theorem}[section]
\newtheorem{cor}[thm]{Corollary}
\newtheorem{lem}[thm]{Lemma}
\newtheorem{prop}[thm]{Proposition}
\newtheorem{defn}[thm]{Definition}
\newtheorem{rem}[thm]{Remark}
\newtheorem{ex}[thm]{Example}

\newtheorem*{thmA}{Theorem A}
\newtheorem*{thmB}{Theorem B}
\newtheorem*{corC}{Corollary C}
\newtheorem*{corD}{Corollary D}
\newtheorem*{propE}{Proposition E}

\begin{document}

\title[Pro-$\mathbf H_p$ density in the free group of rank two]
{Pro-$\mathbf H_p$ density in the free group of rank two: a Frobenian criterion on the mod-$p$ torus}

\author{Jianchun Wu}
\address{Department of Mathematics, Soochow University, Suzhou 215006, CHINA}
\email{wujianchun@suda.edu.cn}

\thanks{The author acknowledges partial support from the National Natural Science Foundation of China (Grant No. 12271385).}

\subjclass[2020]{Primary 20E05; Secondary 20E18, 14G05, 11R45}

\keywords{pro-$\mathbf V$ topology, $\mathbf H_p$-dense subgroup, free group, Laurent polynomial, mod-$p$ torus, Frobenian set}

\begin{abstract}
Let $F$ be the free group of rank two and, for a prime $p$, let
$\mathbf H_p=\mathbf G_p*\mathbf{Ab}_{p-1}$ be the pseudovariety of finite
groups having a normal $p$-subgroup with abelian quotient of exponent
dividing $p-1$. For
$H$ of rank two with $\ab_F(H)=F^{\ab}$ we determine the set
$\mathfrak D(H)$ of primes $p$ where $H$ is $\mathbf H_p$-dense in $F$, and we show that this prime set is
Frobenian in the sense of Serre: it is governed by a single Laurent
polynomial $g$ attached to $H$. We prove that $p\in\mathfrak D(H)$ if and only if $g$ has no zero on the
torus $(\F_p^\times)^2$, and hence that $\mathfrak D(H)$ possesses a
computable natural density $d(H)$.  Exactly one of three cases holds:
$\mathfrak D(H)$ is the set of all primes, it is finite, or it is neither
and $d(H)$ satisfies $\frac1{|G|}\le d(H)\le1-\frac1{|G|}$, where
$G$ is the finite Galois group attached to $g$.
\end{abstract}

\maketitle

\section{Introduction}

Let $F=\langle x_1,x_2\rangle$ be the free group of rank two.  A
pseudovariety $\mathbf V$ of finite groups gives $F$ its
pro-$\mathbf V$ topology, and a subgroup $H\le F$ is \emph{$\mathbf
V$-dense} in $F$ if and only if $HK=F$ for every normal subgroup
$K\trianglelefteq F$ with $F/K\in\mathbf V$; equivalently, $H$ maps
onto every quotient of $F$ in $\mathbf V$.  Determining denseness, and
more generally computing pro-$\mathbf V$ closures of finitely generated
subgroups, has been carried out for several pseudovarieties: for all
finite groups by M.~Hall \cite{Hall49,Hall50}, for finite $p$-groups by
Ribes and Zalesskii \cite{RZ94}, and  for others by Margolis, Sapir and Weil \cite{MSW01} (see also \cite{Weil98}).  These computations are organized around the
Stallings automaton of the subgroup \cite{Stallings,KM02}.

For the pseudovariety $\Gp$ of finite $p$-groups the answer is a condition of
linear algebra over the finite field $\F_p$: a subgroup $S$ of a free group $\Gamma$ is
$\Gp$-dense in $\Gamma$ if and only if
$\ab_\Gamma(S)+p\Gamma^{\ab}=\Gamma^{\ab}$
(Lemma~\ref{lem:pdense}).  Write
\[
\Hp=\Gp*\Abp
\]
for the pseudovariety of finite groups possessing a normal $p$-subgroup
whose quotient is abelian of exponent dividing $p-1$.  Every group in $\Hp$
is supersolvable \cite[Proposition~2.6]{AS}, and these pseudovarieties arise
in the analysis of supersolvable closures in free groups \cite{CW}.
Since $\Hp$ varies with $p$, a single subgroup can be $\Hp$-dense for some primes
and fail to be for others.  The set of primes where it succeeds, written
$\mathfrak D(H)$ (Definition~\ref{defn:D}), is the subject of this paper.
We determine it completely for an arbitrary rank two subgroup $H\le F$ with
$\ab_F(H)=F^{\ab}$, and the answer is that $\mathfrak D(H)$ is an arithmetic object:
the complement of a Frobenian set of primes in the sense of Serre
\cite[Subsection~3.3]{Serre}, governed by one plane curve attached to $H$.

The condition
$\ab_F(H)=F^{\ab}$ makes $H$ dense in every finite abelian
quotient of $F$, so in checking $\Hp$-denseness the abelian half is
automatic and the entire difficulty lies in the commutator direction. 

It is a standard fact that the abelianization $[F,F]^{\ab}$ of the commutator subgroup $[F,F]$ is a free module of rank one over the Laurent polynomial ring $\R=\Z[t_1^{\pm1},t_2^{\pm1}]$, generated by the class $\theta$ of the commutator $[x_1,x_2]$.  The $\R$-submodule of $[F,F]^{\ab}$ generated
by the image of $H\cap[F,F]$ has the form $I_H\cdot\theta$ for a unique
ideal $I_H\subseteq\R$; for $H$ free of rank two with $\ab_F(H)=F^{\ab}$
this ideal is principal, $I_H=(g)$
(Proposition~\ref{prop:principal}).  For each prime $p$, $g$ induces a
function on the torus $(\F_p^\times)^2$
(Subsection~\ref{sec:crt}), and the first main result says that this
function decides whether $p$ lies in $\mathfrak D(H)$.

\begin{thmA}[Theorem~\ref{thm:main}]
Let $H\le F$ be free of rank two with $\ab_F(H)=F^{\ab}$ and $I_H=(g)$.
Then
\[
\mathfrak D(H)=\PP\setminus\PP(g),\qquad
\PP(g)=\bigl\{p:\ g(\zeta)=0\ \text{for some }\zeta\in(\F_p^\times)^2\bigr\},
\]
where $\PP$ denotes the set of all primes.
\end{thmA}

The converse of Theorem~A is the following.

\begin{propE}[Proposition~\ref{prop:realize}, Corollary~\ref{cor:which-sets}]
For every $g\in\R$ with $g(1,1)=1$, there exists an $H\le F$,
free of rank two with $\ab_F(H)=F^{\ab}$, such that $I_H=(g)$.  Hence
the sets of the form $\mathfrak D(H)$ are precisely the sets
$\PP\setminus\PP(g)$ with $g\in\R$, $g(1,1)=1$.
\end{propE}

The second main result determines $\PP(g)$ for a fixed $g$.  This is
pure number theory, and the answer is that $\PP(g)$ is Frobenian in the sense of Serre
\cite{Serre}.

\begin{thmB}[Theorem~\ref{thm:torus}]
Let $g\in\R$ with $g(1,1)=1$.  One can compute a
finite Galois extension $L/\Q$, a conjugacy-stable subset $A_g$ of
$G=\Gal(L/\Q)$, and a finite explicitly enumerated set $\mathcal S(g)$ of primes, such
that every $p\notin \mathcal S(g)$ is unramified in $L/\Q$ and satisfies
\[
p\in\PP(g)\iff \Frob_{\mathfrak p}\in A_g
\quad\text{for some, equivalently every, }\mathfrak p\mid p .
\]
\end{thmB}

Theorem~B and the Chebotarev density
theorem give a rigid trichotomy.

\begin{corC}[Theorem~\ref{thm:torus}, Corollary~\ref{cor:trichotomy}]
$\mathfrak D(H)$ is Frobenian and possesses a natural density
$d(H)\in\Q\cap[0,1]$, computable from $H$; and exactly one of the following
holds:
\begin{enumerate}[label=(\roman*)]
\item $\mathfrak D(H)=\PP$, so $d(H)=1$;
\item $\mathfrak D(H)$ is finite, so $d(H)=0$;
\item $\mathfrak D(H)$ is neither finite nor cofinite, and
$\frac1{|G|}\le d(H)\le1-\frac1{|G|}$.
\end{enumerate}
\end{corC}

\begin{corD}[Corollary~\ref{cor:decide}]
The trichotomy of Corollary~C is computable: there is an algorithm
which, given two words in $x_1^{\pm1},x_2^{\pm1}$ forming a basis of
such an $H$, computes $\mathfrak D(H)$, decides which of the cases
(i)--(iii) holds, lists $\mathfrak D(H)$ when it is finite, and
computes $d(H)$ as an explicit rational number.
\end{corD}

The paper is organized as follows.  Section~\ref{sec:prelim} fixes
notation and introduces the technical ingredients used throughout: the
characteristic subgroups $N_k$, the structure of $[F,F]^{\ab}$ as an
$\R$-module, the ring $\R_{p,p-1}$, and the pro-$\mathbf V$ topologies.
Section~\ref{sec:density} constructs the invariant $I_H$ and
proves Theorem~A and Proposition~E.
Section~\ref{sec:dec} proves Theorem~B and deduces
Corollaries~C and~D.  Appendix~\ref{app:ruppert} records Ruppert's
criterion under reduction, and Appendix~\ref{app:radical} the radical
computation used for the certificates.

\section{Preliminaries}
\label{sec:prelim}

\subsection{The characteristic subgroups $N_k$}
For a group $G$, let $\rk(G)$ denote its rank and $G^{\ab}=G/[G,G]$ its abelianization,
with quotient map $\ab_G:G\to G^{\ab}$. We write
$F^{\ab}=\Z^2$ with standard basis $e_1=\ab_F(x_1),e_2=\ab_F(x_2)$.
Throughout we adopt the standard conventions
\[
v^y = y^{-1}vy ,\qquad
[y_1,y_2] = y_1^{-1}y_2^{-1}y_1y_2 .
\]

For a positive integer $k$ let
\begin{equation}
\label{eq:Nk}
N_k=\ker\Bigl(F\xrightarrow{\ \ab_F\ }\Z^2\xrightarrow{\bmod k}\Z_k^2\Bigr)
=[F,F]F^k,\qquad F^k=\langle y^k:y\in F\rangle.
\end{equation}
Then $N_k$ is a characteristic subgroup of $F$ of index $k^2$,
and $[F,F]\subseteq N_k$.

For a subgroup $H\le F$ we write
\[
A_H=\ab_F(H)\subseteq F^{\ab},\qquad C_H=H\cap[F,F].
\]

\begin{lem}
\label{lem:hcapnk}
Let $H\le F$  and let $b_1,\dots,b_m$ be a
$\Z$-basis of $A_H\cap kF^{\ab}$. If $h_j\in H$ satisfies
$\ab_F(h_j)=b_j$, then
\[
H\cap N_k=\bigl\langle h_1,\dots,h_m,\;C_H\bigr\rangle.
\]
\end{lem}

\begin{proof}
Each $h_j$ lies in $N_k$, because $\ab_F(h_j)=b_j\in kF^{\ab}$, and
$H\cap[F,F]\subseteq N_k$ by \eqref{eq:Nk}; hence all generators of the
right-hand side lie in $H\cap N_k$. Conversely, let
$h\in H\cap N_k$. Then $\ab_F(h)\in A_H\cap kF^{\ab}$, so for suitable
$n_1,\dots,n_m\in\Z$ we have $\ab_F(h)=\sum_j n_j b_j$ and
$\ab_F\bigl(h\,\prod_j h_j^{-n_j}\bigr)=0$; since both factors lie in $H$, this
element lies in $H\cap[F,F]=C_H$. As $\prod_j h_j^{-n_j}\in\langle h_1,\dots,h_m\rangle$,
we get $h\in\langle h_1,\dots,h_m,C_H\rangle$.
\end{proof}

\begin{lem}
\label{lem:normalized-basis}
Let $H\le F$ have rank two with $A_H=F^{\ab}$.  Then $C_H=[H,H]$, and $H$ has a free basis
$h_1,h_2$ with $\ab_F(h_i)=e_i$ such that for every positive integer $k$
\[
N_{k}=\bigl\langle h_1^{k},\,h_2^{k},\,[F,F]\bigr\rangle,
\qquad
H\cap N_{k}=\bigl\langle h_1^{k},\,h_2^{k},\,C_H\bigr\rangle.
\]
\end{lem}

\begin{proof}
As $A_H$ is abelian, the restriction $\ab_F|_H: H \to A_H$ factors through $\ab_H: H\to H^{\ab}$: there is a surjection $\phi:H^{\ab}\to A_H$ with $\ab_F|_H=\phi\circ \ab_H$. Hence
$$\ker \phi=\ab_H(\ker \ab_F|_H)=\ab_H(C_H)=C_H/[H,H].$$
Both $H^{\ab}$ and $A_H$ are free abelian of rank two, so this surjection is an isomorphism and $\ker\phi=0$, i.e.\ $C_H=[H,H]$.

Since the natural map $\mathrm{Aut}(H)\to\mathrm{Aut}(H^{\ab})\cong\mathrm{GL}_2(\Z)$ is surjective, the $\Z$-basis $\{\phi^{-1}(e_1),\phi^{-1}(e_2)\}$ of $H^{\ab}$ is induced by a free basis $\{h_1,h_2\}$ of $H$ with $\ab_H(h_i)=\phi^{-1}(e_i)$; then $\ab_F(h_i)=\phi\circ\ab_H(h_i)=e_i$.

Since $N_{k}/[F,F]=kF^{\ab}$ is freely generated by
$ke_1,ke_2$ and $\ab_F(h_i^{k})=ke_i$, the subgroup
$\langle h_1^{k},h_2^{k}\rangle$ maps onto
$kF^{\ab}$; as $[F,F]$ is the kernel of $F\twoheadrightarrow F^{\ab}$, we
get $N_{k}=\langle h_1^{k},h_2^{k},[F,F]\rangle$.  Applying
Lemma~\ref{lem:hcapnk} to the basis $\{b_1,b_2\}=\{ke_1,ke_2\}$ of
$A_H\cap kF^{\ab}=k\Z^2$ with the lifts $h_1^{k},h_2^{k}$
 gives the second identity.
\end{proof}

\subsection{The \texorpdfstring{$\R$}{R}-module \texorpdfstring{$[F,F]^{\ab}$}{[F,F]ab}}
\label{sec:tomaszewski}
For a subgroup $K\trianglelefteq F$ such that $[F,F]\subseteq K$, the conjugation action of $F$ on $K$ induces an action on its
abelianization $K^{\ab}$. For $y\in F$ and $v\in  K$ the action is given by
\[
y\cdot\ab_{K}(v) = \ab_{K}(y^{-1}vy).
\]
If $y\in[F,F]$, then $y^{-1}vy = v[v,y]$, so $y\cdot\ab_{K}(v)=\ab_{K}(v)$ since $[v,y]\in [K,K]$;
hence the action induces an action of $F^{\ab}$ on $K^{\ab}$, making
$K^{\ab}$ a module over the integral group ring $\Z[F^{\ab}]$.

Let $\R$ denote the integral Laurent polynomial ring $\Z[t_1^{\pm1},t_2^{\pm1}]$.  
$K^{\ab}$  can be viewed as  an $\R$-module as follows: for a fixed basis, for instance,  $\{e_1,e_2\}$ of $F^{\ab}$,
the map 
$$\kappa_F: F^{\ab}\to \R, \qquad a_1 e_1+a_2 e_2\mapsto t_1^{a_1}t_2^{a_2}$$ can be extended linearly to a ring isomorphism
$\kappa_F:\Z[F^{\ab}]\cong \R$. The action of $\R$
on $K^{\ab}$ is defined via $\kappa_F$.
Explicitly,  for a monomial $t_1^{a_1}t_2^{a_2}\in \R$, let $y\in F$ with $\kappa_F(\ab_F(y))=t_1^{a_1}t_2^{a_2}$, so that
\[
t_1^{a_1}t_2^{a_2}\cdot\ab_{K}(v) =\ab_F(y)\cdot \ab_{K}(v)=
 \ab_{K}(y^{-1}vy).
\]

Throughout we write
$\overline{v}= \ab_{[F,F]}(v)$ for $v\in [F,F]$.

\begin{lem}[{\cite[p.~2]{Put}}]
\label{lem:rankone}
$[F,F]^{\ab}$ is a free $\R$-module of rank one, generated by
$\theta=\overline{[x_1,x_2]}$.
\end{lem}

\subsection{The ring \texorpdfstring{$\R_{p,p-1}$}{Rppminus1}}
\label{sec:crt}

For a prime $p$, write $\F_p$ for the field with $p$ elements and $\F_p^\times$ for its multiplicative group.  Let
\[
\R_{p-1}=\Z[t_1^{\pm1},t_2^{\pm1}]\big/(t_1^{p-1}-1,\;t_2^{p-1}-1),\qquad
\R_{p,p-1}=\R_{p-1}/p\R_{p-1}.
\]

An element of $\R$ is a finite sum
$\sum_{a,b\in\Z}c_{a,b}t_1^a t_2^b$ with $c_{a,b}\in\Z$.  In $\R_{p-1}$ the
relations $t_i^{p-1}-1=0$ give $t_i^a=t_i^{\,a\bmod(p-1)}$; in particular
$t_i^{-1}=t_i^{\,p-2}$, so $t_i$ is already invertible in $\R_{p-1}$, hence
\[
\R_{p-1}=\Z[t_1,t_2]\big/(t_1^{p-1}-1,\;t_2^{p-1}-1),
\]
a free $\Z$-module on the $(p-1)^2$ monomials $t_1^at_2^b$ with
$0\le a,b\le p-2$.  Reducing the coefficients modulo $p$, we have
\begin{equation}
\label{eq:Rbasis}
\R_{p,p-1}=\F_p[t_1,t_2]\big/(t_1^{p-1}-1,\;t_2^{p-1}-1)=\bigoplus_{a,b=0}^{p-2}\F_p\,t_1^at_2^b ,
\qquad \dim_{\F_p}\R_{p,p-1}=(p-1)^2 .
\end{equation}

\begin{lem}
\label{lem:crt}
The evaluation map
\[
\mathrm{ev}:\R_{p,p-1}\;\longrightarrow\;\prod_{\zeta\in(\F_p^\times)^2}\F_p,
\qquad f\mapsto\bigl(f(\zeta)\bigr)_\zeta ,
\]
is an isomorphism of $\F_p$-algebras.
\end{lem}

\begin{proof}
Each component $\mathrm{ev}_\zeta:f\mapsto f(\zeta)$ is well defined: the ring
homomorphism $\F_p[t_1,t_2]\to\F_p$ sending $t_i\mapsto\zeta_i$ maps
$t_i^{p-1}-1$ to 0 because $\zeta_i^{p-1}=1$, so it factors through
$\R_{p,p-1}$.  Thus $\mathrm{ev}$ is a homomorphism of $\F_p$-algebras; in
particular $(fh)(\zeta)=f(\zeta)h(\zeta)$.

For $\gamma\in\F_p^\times$,  one has $\sum_{r=0}^{p-2}\gamma^{\,r}=p-1=-1$ if
$\gamma=1$, and $\sum_{r=0}^{p-2}\gamma^{\,r}=0$ otherwise, since
$(\gamma-1)\sum_{r=0}^{p-2}\gamma^{\,r}=\gamma^{p-1}-1=0$.  Given
$\zeta\in(\F_p^\times)^2$, let
\[
e_\zeta=\prod_{i=1,2}\ \sum_{r=0}^{p-2}\zeta_i^{-r}\,t_i^{\,r}
\ \in\ \R_{p,p-1} .
\]
Evaluating factor by factor at $\eta$ and setting $\gamma_i=\zeta_i^{-1}\eta_i$,
\[
e_\zeta(\eta)=\prod_{i=1,2}\ \sum_{r=0}^{p-2}\gamma_i^{\,r},
\]
The $i$th factor is $-1$ if $\eta_i=\zeta_i$ and $0$ otherwise; hence the product
is $(-1)^2=1$ if $\eta=\zeta$ and $0$ if $\eta\ne\zeta$.  So the image of
$\mathrm{ev}$ contains every standard basis vector of the product, and
$\mathrm{ev}$ is surjective.  Both sides have $\F_p$-dimension $(p-1)^2$ by
\eqref{eq:Rbasis}, so it is an isomorphism.
\end{proof}

For $i=1,2$ let $\Sigma_i=\sum_{r=0}^{p-2}t_i^{\,r}\in\R_{p,p-1}$; then
$\Sigma_1\Sigma_2$ is the element $e_{(1,1)}$ of the preceding proof, so
\begin{equation}
\label{eq:Sigma-idem}
(\Sigma_1\Sigma_2)(1,1)=1,\qquad
(\Sigma_1\Sigma_2)(\zeta)=0\quad\text{for }\zeta\ne(1,1);
\end{equation}
that is, $\mathrm{ev}$ carries $\Sigma_1\Sigma_2$ to the standard basis vector of
the product indexed by $(1,1)$.

\begin{cor}
\label{cor:JSigma}
The ideal $J=(\Sigma_1\Sigma_2)$ of $\R_{p,p-1}$ equals
$\F_p\cdot\Sigma_1\Sigma_2$, so $\dim_{\F_p}J=1$, and $\mathrm{ev}$ induces
an isomorphism
\[
\R_{p,p-1}/(\Sigma_1\Sigma_2)\;\cong
\prod_{\zeta\in(\F_p^\times)^2\setminus\{(1,1)\}}\F_p ,
\qquad
\dim_{\F_p}\R_{p,p-1}/J=(p-1)^2-1 .
\]
\end{cor}

\begin{proof}
By \eqref{eq:Sigma-idem}, $\mathrm{ev}$ sends the generator $\Sigma_1\Sigma_2$
of $J$ to the standard basis vector indexed by $(1,1)$; as $\mathrm{ev}$ is an
isomorphism (Lemma~\ref{lem:crt}), $J$ is the preimage of the corresponding
coordinate factor $\F_p$, that is, $J=\F_p\cdot\Sigma_1\Sigma_2$ and
$\dim_{\F_p}J=1$.  Passing to the quotient, $\mathrm{ev}$ induces
$\R_{p,p-1}/J\cong\prod_{\zeta\ne(1,1)}\F_p$, of dimension $(p-1)^2-1$.
\end{proof}

In Section~\ref{sec:density} the above is applied to $g\in\R$.  The image
$\bar g\in\R_{p,p-1}$ is by Lemma~\ref{lem:crt} determined by its values.
Write
\[
g(\zeta)=\bar g(\zeta),\qquad \zeta\in\F_p^2 ,
\]
for the value of the reduction $\bar g$ at $\zeta$; saying $g$ has no zero
in $(\F_p^\times)^2$ means $g(\zeta)\ne0$ for every $\zeta\in(\F_p^\times)^2$.

\subsection{Pseudovarieties and the
\texorpdfstring{pro-$\mathbf V$}{pro-V} topology}

\begin{defn}
A class of finite groups is a \emph{pseudovariety} if it is closed
under taking subgroups, homomorphic images, and finite direct
products. The \emph{product} $\mathbf{U}*\mathbf{V}$ of two
pseudovarieties $\mathbf{U}$, $\mathbf{V}$ consists of all groups $G$
having a normal subgroup $N\in\mathbf{U}$ such that $G/N\in\mathbf{V}$.
\end{defn}
The pseudovarieties in this paper are:
$\mathbf{G}_p$ (finite $p$-groups),
$\mathbf{Ab}_d$ (finite abelian groups of exponent dividing $d$),
 and
$\mathbf{H}_p=\mathbf{G}_p*\mathbf{Ab}_{p-1}$.
%

For a pseudovariety $\mathbf V$, the pro-$\mathbf{V}$ topology on $F$
is the coarsest topology such that every homomorphism from $F$ to a
finite group (endowed with discrete topology) in $\mathbf{V}$ is
continuous. Let $\Cl_{\mathbf V}(S)$ denote the closure of
$S\subseteq F$ in the pro-$\mathbf V$ topology. For $S\leq \Gamma\leq F$, let
$\Cl_{\mathbf V}(S,\Gamma)$ denote the closure of $S$ inside $\Gamma$ with
respect to the pro-$\mathbf V$ topology on $\Gamma$ (it may differ from the
subspace topology). $S$ is $\mathbf V$-dense in $\Gamma$ if
$\Cl_{\mathbf V}(S,\Gamma)=\Gamma$.

By \cite[Corollary~3.2]{MSW01}, a finitely generated subgroup of a free
group is $\Gp$-dense in $\Gamma$ exactly when it maps onto the elementary
abelian quotient $\Gamma^{\ab}/p\Gamma^{\ab}$.  This is the following
linear criterion.

\begin{lem}
\label{lem:pdense}
Let $\Gamma$ be a free group of finite rank and let $S\le\Gamma$ be a
finitely generated subgroup.  Then
\[
\Cl_{\Gp}(S,\Gamma)=\Gamma\qquad\Longleftrightarrow\qquad
\ab_\Gamma(S)+p\Gamma^{\ab}=\Gamma^{\ab}.
\]
\end{lem}

Lemma~\ref{lem:pdense} is only the denseness criterion; computing the whole
pro-$\mathbf{G}_p$ closure of a finitely generated subgroup of a free group is the
 deeper theorem of Ribes and Zalesskii \cite{RZ94}, see also
\cite{MSW01}.  As in the $\Gp$ case, $\Hp$-denseness is governed by an
analogous criterion.

\begin{lem}[{\cite[Lemma~3.11]{CW}}]
\label{lem:decomp}
Let $H\subseteq F$ be finitely generated. Then $H$ is $\mathbf{H}_p$-dense in $F$ if and only if
$H$ is $\Ab_{p-1}$-dense in $F$ and $H\cap N_{p-1}$ is $\Gp$-dense in $N_{p-1}$.
\end{lem}

Let $\PP$ denote the set of all rational primes.

\begin{defn}
\label{defn:D}
For a finitely generated subgroup $H\le F$ let
\[
\mathfrak D(H)=\bigl\{p\in\PP:\ H\ \text{is }\Hp\text{-dense in }F\bigr\} .
\]
\end{defn}

\section{The Density Criterion}
\label{sec:density}

Let
$
\mathfrak C_H=\R\cdot\ab_{[F,F]}(C_H),
$
be the $\R$-submodule of $[F,F]^{\ab}$ generated by $C_H$. By
Lemma~\ref{lem:rankone}, the map $r\mapsto r\cdot\theta$ is an isomorphism $\R\to[F,F]^{\ab}$,
so submodules of $[F,F]^{\ab}$ correspond bijectively to ideals of $\R$. Hence there exists a unique ideal $I_H\subseteq\R$ such that
\[
\mathfrak C_H=I_H\cdot\theta.
\]

\begin{lem}
\label{lem:Rclosed}
Assume $A_H=F^{\ab}$. Then
$\ab_{[F,F]}(C_H)=\mathfrak C_H$.
\end{lem}

\begin{proof}
For $v\in C_H$ and a monomial $t_1^{a_1}t_2^{a_2}\in\R$, since $A_H=F^{\ab}$ there is $y\in H$ with $\kappa_F(\ab_F(y))=t_1^{a_1}t_2^{a_2}$. Then
\[
t_1^{a_1}t_2^{a_2}\cdot\ab_{[F,F]}(v)=\overline {y^{-1}vy}\in \ab_{[F,F]}(C_H),
\]
the membership because $y^{-1}vy\in C_H$, $C_H$ being normal in $H$.
Hence $\mathfrak C_H\subseteq \ab_{[F,F]}(C_H)$, and the reverse inclusion is immediate.
\end{proof}

Call a free basis $h_1,h_2$ of $H$ \emph{normalized} if $\ab_F(h_i)=e_i$ for
$i=1,2$; by Lemma~\ref{lem:normalized-basis} a normalized basis exists as soon
as $H$ has rank two and $A_H=F^{\ab}$.

\begin{prop}
\label{prop:principal}
Let $H\le F$ be free of rank two with $A_H=F^{\ab}$, and let $h_1,h_2$ be a
normalized basis of $H$.  Let $g\in\R$ be the element determined by
\[
\overline{[h_1,h_2]}=g\cdot\theta ,
\]
which exists and is unique because $[F,F]^{\ab}$ is free of rank one on
$\theta$ (Lemma~\ref{lem:rankone}).  Then
\[
\mathfrak C_H=(g)\cdot\theta,\qquad
I_H=(g),\qquad
g(1,1)=1 .
\]
\end{prop}

\begin{proof}

As $\ab_F(h_i)=e_i$, write $h_i=x_iv_i$ with $v_i\in [F,F]$:
\[
[h_1,h_2]=[x_1v_1,x_2v_2]
=\bigl([x_1,v_2]\,[x_1,x_2]^{v_2}\bigr)^{v_1}\,
[v_1,v_2]\,[v_1,x_2]^{v_2}.
\]
Applying $\ab_{[F,F]}$, and using $\ab_{[F,F]}([v_1,v_2])=0$, we get
\[
\overline{[h_1,h_2]}
=\theta+\overline{[x_1,v_2]}+\overline{[v_1,x_2]}=\theta+(1-t_1)\cdot\overline{v_2}+(t_2-1)\cdot\overline{v_1}.
\]
Write $\overline{v_i}=s_i\cdot\theta$ with $s_i\in\R$, so that 
\[
\overline{[h_1,h_2]}
=(1+(1-t_1)s_2+(t_2-1)s_1)\cdot\theta.
\]
Therefore $g=1+(1-t_1)s_2+(t_2-1)s_1$, and $g(1,1)=1$.

By
Lemma~\ref{lem:Rclosed}, $\ab_{[F,F]}(C_H)=\mathfrak C_H$.  Since
$C_H=[H,H]$ is generated as a group by the conjugates $[h_1,h_2]^{h}$
($h\in H$), for each such conjugate
\[
\ab_{[F,F]}\bigl([h_1,h_2]^{h}\bigr)
=\ab_F(h)\cdot\ab_{[F,F]}([h_1,h_2])
=\kappa_F(\ab_F(h))\cdot(g\cdot\theta).
\]
As $A_H=F^{\ab}$, the elements $\kappa_F(\ab_F(h))$, $h\in H$, range over
all monomials $t_1^{a_1}t_2^{a_2}$, so
$\mathfrak C_H=\ab_{[F,F]}(C_H)=\R\cdot(g\cdot\theta)=(g)\cdot\theta$ and $I_H=(g)$.
\end{proof}

\begin{rem}
\label{rem:g-depends-on-basis}
The element $g$ depends on the chosen basis, whereas the ideal $I_H$  is the invariant of $H$.  Indeed, let
$\{h_1',h_2'\}$ be any basis of $H$.  By Nielsen's theorem every
automorphism of a free group of rank two carries $[h_1,h_2]$ to a conjugate
of $[h_1,h_2]^{\pm1}$ (\cite[Theorem~3.9]{MKS}); so
$[h_1',h_2']=\bigl([h_1,h_2]^{\pm 1}\bigr)^{h}$ for some $h\in H$.  Applying $\ab_{[F,F]}$ and
writing $t_1^at_2^b=\kappa_F(\ab_F(h))$ gives
\[
\ab_{[F,F]}\bigl([h_1',h_2']\bigr)=\pm\,t_1^{a}t_2^{b}\cdot(g\cdot\theta) .
\]
Since $\pm t_1^at_2^b$ is a unit of $\R$, the element $g'$
determined by $\ab_{[F,F]}([h_1',h_2'])=g'\cdot\theta$ satisfies
$(g')=(g)=I_H$.  Multiplying by a monomial to clear negative exponents, dividing out the largest monomial factor common to all terms, and then multiplying by $-1$ if necessary, we may
assume
\begin{equation}
\label{eq:normalized}
g\in\Z[t_1,t_2],\qquad t_1\nmid g,\quad t_2\nmid g,\qquad g(1,1)=1 ,
\end{equation}
the last condition being automatic, a monomial taking the value $1$ at $(1,1)$.  Call such a $g$ the \emph{normalized} generator of
$I_H$. It is unique: the units of $\R$ being $\pm t_1^at_2^b$, two such
generators differ by $\pm t_1^at_2^b$, and \eqref{eq:normalized} forces
$a=b=0$ and the sign $+$; hence $\deg g$ is an invariant of $H$.
\end{rem}

The above $g$ is subject to no constraint beyond
$g(1,1)=1$: every such element is realized by some subgroup $H$ with $I_H=(g)$.

\begin{prop}
\label{prop:realize}
Let $g\in\R$ satisfy $g(1,1)=1$.  Choose $s_1,s_2\in\R$ with
\begin{equation}
\label{eq:solve-g}
g-1=(1-t_1)s_2+(t_2-1)s_1
\end{equation}
and $v_1,v_2\in[F,F]$ with $\ab_{[F,F]}(v_i)=s_i\cdot\theta$, and let
$h_i=x_iv_i$ and $H=\langle h_1,h_2\rangle$.  Then $H$ is free of rank two
with $A_H=F^{\ab}$, the pair $h_1,h_2$ is a basis of $H$ with
$\ab_F(h_i)=e_i$, and $I_H=(g)$.
\end{prop}

\begin{proof}
The evaluation $\mathrm{ev}_{\R}:\R\to\Z$, $t_1,t_2\mapsto1$, is the
augmentation of the group ring $\Z[F^{\ab}]$ under $\kappa_F$, so its kernel
is the augmentation ideal $(t_1-1,t_2-1)$.  As $\mathrm{ev}_{\R}(g-1)=0$, a pair
$s_1,s_2$ as in \eqref{eq:solve-g} exists; and $v_1,v_2$ exist because
$\ab_{[F,F]}$ maps $[F,F]$ onto $[F,F]^{\ab}=\R\cdot\theta$.

Since $v_i\in[F,F]$ we have $\ab_F(h_i)=\ab_F(x_i)=e_i$, so
$A_H=\ab_F(H)$ contains the basis $e_1,e_2$ of $F^{\ab}$ and hence equals
$F^{\ab}$.  In particular $H$ is not cyclic, so is free of rank two, and $h_1,h_2$ is a normalized basis of it.  The
computation in the proof of
Proposition~\ref{prop:principal}, carried out with these $v_i$, gives
$\ab_{[F,F]}([h_1,h_2])=(1+(1-t_1)s_2+(t_2-1)s_1)\cdot\theta$, that is $\ab_{[F,F]}([h_1,h_2])=g\cdot\theta$ by
\eqref{eq:solve-g}.  Proposition~\ref{prop:principal} then yields
$I_H=(g)$.
\end{proof}
Throughout $H\le F$ is \emph{free of rank two with $A_H=\ab_F(H)=F^{\ab}$}. We
fix a normalized basis $h_1,h_2$ of $H$ and let $g\in\R$ be the associated
element, so that $\ab_{[F,F]}([h_1,h_2])=g\cdot\theta$, $I_H=(g)$ and
$g(1,1)=1$ as in Proposition~\ref{prop:principal}.

Fix a prime $p$, and let $S=H\cap N_{p-1}$.  By Lemma~\ref{lem:normalized-basis} with $k=p-1$
\[
S=\langle h_1^{p-1},h_2^{p-1},C_H\rangle,\qquad
N_{p-1}=\langle h_1^{p-1},h_2^{p-1},[F,F]\rangle .
\]
For $p=2$ this reads $S=H$ and $N_1=F$, and all statements below hold in
the resulting degenerate form (e.g.\ $V=W=0$ in Lemma~\ref{lem:VW}).

\begin{lem}
\label{lem:decompAb}
Both of the following are direct sums:
\begin{enumerate}[label=(\arabic*), ref=(\arabic*)]
\item\label{dab:N}
$N_{p-1}^{\ab}=\ab_{N_{p-1}}\bigl(\langle h_1^{p-1},h_2^{p-1}\rangle\bigr)
\oplus \ab_{N_{p-1}}([F,F])$;
\item\label{dab:S}
$\ab_{N_{p-1}}(S)=\ab_{N_{p-1}}\bigl(\langle h_1^{p-1},h_2^{p-1}\rangle\bigr)
\oplus \ab_{N_{p-1}}(C_H)$.
\end{enumerate}
\end{lem}

\begin{proof}
\ref{dab:N} Let $Q=N_{p-1}/[F,F]$. Since
$[N_{p-1},N_{p-1}]\le[F,F]$, the quotient map $N_{p-1}\twoheadrightarrow Q$
factors through the abelianization, yielding a canonical projection
$q:N_{p-1}^{\ab}\twoheadrightarrow Q$.  By Lemma~\ref{lem:normalized-basis},
$N_{p-1}=\langle  h_1^{p-1},h_2^{p-1},[F,F]\rangle$; hence $Q$ is freely generated by
$q(\ab_{N_{p-1}}(h_1^{p-1})),q(\ab_{N_{p-1}}(h_2^{p-1}))$.

The generators listed above give
$N_{p-1}^{\ab}=\ab_{N_{p-1}}(\langle  h_1^{p-1},h_2^{p-1}\rangle)+\ab_{N_{p-1}}([F,F])$.
If $w$ lies in the intersection of these two subgroups, write
\[w=a\cdot\ab_{N_{p-1}}(h_1^{p-1})+b\cdot\ab_{N_{p-1}}(h_2^{p-1})=\ab_{N_{p-1}}(v)\] with
$a,b\in\Z$ and $v\in[F,F]$.  Then
$a\cdot q(\ab_{N_{p-1}}(h_1^{p-1}))+b\cdot q(\ab_{N_{p-1}}(h_2^{p-1}))=q(\ab_{N_{p-1}}(v))=0$. Hence $a=b=0$ and $w=0$, so the intersection is trivial and
the sum is direct.

\ref{dab:S} From $S=\langle  h_1^{p-1},h_2^{p-1},C_H\rangle$ we get
$\ab_{N_{p-1}}(S)=\ab_{N_{p-1}}(\langle  h_1^{p-1},h_2^{p-1}\rangle)+\ab_{N_{p-1}}(C_H)$.  Since
$\ab_{N_{p-1}}(C_H)\subseteq\ab_{N_{p-1}}([F,F])$, the two summands lie in
the two summands of \ref{dab:N}, so the sum is again direct.
\end{proof}

Since $N_{p-1}^{\ab}$ is an abelian group, the restriction of $\ab_{N_{p-1}}$ to $[F,F]$ factors through $\ab_{[F,F]}$,  giving the  following commutative diagram:
\begin{equation}
\label{tonab-k}
\begin{tikzcd}[column sep=large]
    {[F,F]} \ar[r, "{\ab_{[F,F]}}"] \ar[hook, d] & {[F,F]^{\ab}} \ar[d, "\iota"] \\
    N_{p-1} \ar[r, "\ab_{N_{p-1}}"'] & N_{p-1}^{\ab}
\end{tikzcd}
\end{equation}

\begin{lem}
\label{lem:iota-linear}
The canonical map $\iota$ of \eqref{tonab-k},
\[
\iota:[F,F]^{\ab}\longrightarrow N_{p-1}^{\ab},
\qquad
\iota\bigl(\ab_{[F,F]}(w)\bigr)=\ab_{N_{p-1}}(w)\quad(w\in[F,F]),
\]
induced by the inclusion $[F,F]\hookrightarrow N_{p-1}$, is $\R$-linear with respect
to the $\R$-module structures of Subsection~\ref{sec:tomaszewski}.
\end{lem}

\begin{proof}
Since $\R$ is generated
over $\Z$ by the monomials and both sides are $\Z$-linear, it suffices to prove  \[\iota(t_1^at_2^b\cdot m)=t_1^at_2^b\cdot\iota(m)\]
for every monomial and $m\in [F,F]^{\ab}$.

Assume $m=\ab_{[F,F]}(v)$.  The left-hand side is:
\[
\begin{aligned}
\iota\bigl(t_1^at_2^b\cdot\ab_{[F,F]}(v)\bigr)
&=\iota\Bigl(\ab_{[F,F]}\bigl((x_1^{a}x_2^{b})^{-1}v\,x_1^{a}x_2^{b}\bigr)\Bigr)\\
&=\ab_{N_{p-1}}\bigl((x_1^{a}x_2^{b})^{-1}v\,x_1^{a}x_2^{b}\bigr).
\end{aligned}
\]
The right-hand side is:
\[
t_1^at_2^b\cdot\iota\bigl(\ab_{[F,F]}(v)\bigr)
=t_1^at_2^b\cdot\ab_{N_{p-1}}(v)
=\ab_{N_{p-1}}\bigl((x_1^{a}x_2^{b})^{-1}v\,x_1^{a}x_2^{b}\bigr).
\]
The two results coincide, so the lemma follows.
\end{proof}

\begin{lem}
\label{lem:ppminus1-mod}
$N_{p-1}^{\ab}$ is an $\R_{p-1}$-module. Both $\ab_{N_{p-1}}([F,F])$ and $\ab_{N_{p-1}}(C_H)$ are cyclic submodules, generated by $\ab_{N_{p-1}}([x_1,x_2])$ and $\ab_{N_{p-1}}([h_1,h_2])$  respectively. Moreover $\ab_{N_{p-1}}([x_1,x_2])=\iota(\theta)$ and $\ab_{N_{p-1}}([h_1,h_2])=\iota(g\cdot\theta)$.
\end{lem}

\begin{proof}
 By
Section~\ref{sec:tomaszewski}, $N_{p-1}^{\ab}$ carries an $\R$-module structure:
a monomial $t_1^{a_1}t_2^{a_2}\in\R$ acts on a class $\ab_{N_{p-1}}(v)\in N_{p-1}^{\ab}$
(represented by an element $v\in N_{p-1}$) by 
\[
t_1^{a_1}t_2^{a_2}\cdot \ab_{N_{p-1}}(v)=\ab_{N_{p-1}}(y^{-1}vy)
\]
where $y\in F$ is any lift satisfying
$\ab_F(y)=a_1e_1+a_2e_2$.

 For every $v\in N_{p-1}$, since $x_i^{p-1}\in N_{p-1}$,  $[v,x_i^{p-1}]$ lies in $[N_{p-1},N_{p-1}]$, and
\[
t_i^{p-1}\cdot\ab_{N_{p-1}}(v)=\ab_{N_{p-1}}(x_i^{-(p-1)}\,v\,x_i^{p-1})=\ab_{N_{p-1}}(v[v,x_i^{p-1}])=\ab_{N_{p-1}}(v).
\]
Hence $t_i^{p-1}$ acts as the identity on $N_{p-1}^{\ab}$, and the $\R$-action
depends only on the exponent of $t_i$ modulo $p-1$.  The action therefore factors
through the quotient $\R_{p-1}=\R/(t_1^{p-1}-1,t_2^{p-1}-1)$, so $N_{p-1}^{\ab}$
is an $\R_{p-1}$-module.  For every $m\in N_{p-1}^{\ab}$ we then have
$\R\cdot m=\R_{p-1}\cdot m$; hence a cyclic $\R$-submodule on a given generator
is equally the cyclic $\R_{p-1}$-submodule on the same generator.

By Lemma~\ref{lem:rankone}, $[F,F]^{\ab}$ is a free rank-one
$\R$-module on $\theta=\overline{[x_1,x_2]}$.  As $\iota$ is $\R$-linear (Lemma~\ref{lem:iota-linear}), its image
$\ab_{N_{p-1}}([F,F])$ is the cyclic $\R_{p-1}$-submodule
generated by
\[
\iota(\theta)=\iota\bigl(\ab_{[F,F]}([x_1,x_2])\bigr)=\ab_{N_{p-1}}([x_1,x_2]).
\]

By Lemma~\ref{lem:Rclosed}, $\ab_{[F,F]}(C_H)=\mathfrak C_H=\R\cdot\ab_{[F,F]}(C_H)$;
by Proposition~\ref{prop:principal}, $\mathfrak C_H=(g)\cdot\theta$.  Hence
$\ab_{[F,F]}(C_H)=\R\cdot(g\cdot\theta)$, the cyclic $\R$-submodule generated by $g\cdot\theta$.
Applying the $\R$-linear $\iota$,
$\ab_{N_{p-1}}(C_H)=\iota(\ab_{[F,F]}(C_H))$ is the cyclic $\R_{p-1}$-submodule
generated by $\iota(g\cdot\theta)$.
\end{proof}

Let $\rho\colon N_{p-1}^{\ab}\twoheadrightarrow N_{p-1}^{\ab}/pN_{p-1}^{\ab}$ be the
quotient map and write $M=\rho(N_{p-1}^{\ab})$.  By Lemma~\ref{lem:pdense},
\[
\Cl_{\Gp}(S,N_{p-1})=N_{p-1}
\iff \ab_{N_{p-1}}(S)+pN_{p-1}^{\ab}=N_{p-1}^{\ab}
\iff \rho(\ab_{N_{p-1}}(S))=M .
\]
Let $W=\rho\bigl(\ab_{N_{p-1}}(C_H)\bigr)$ and
$V=\rho\bigl(\ab_{N_{p-1}}([F,F])\bigr)$.  By Lemma~\ref{lem:decompAb} the two
decompositions are compatible:
\[
\rho(\ab_{N_{p-1}}(S))
=\rho\bigl(\ab_{N_{p-1}}(\langle h_1^{p-1},h_2^{p-1}\rangle)\bigr)\oplus W,
\qquad
M
=\rho\bigl(\ab_{N_{p-1}}(\langle h_1^{p-1},h_2^{p-1}\rangle)\bigr)\oplus V,
\]
so
\begin{equation}
\label{eq:WV}
\Cl_{\Gp}(S,N_{p-1})=N_{p-1}\iff W=V.
\end{equation}

\begin{lem}
\label{lem:VW}
$M$ is a module over $\R_{p,p-1}$; it has two distinguished cyclic submodules. The submodule $V$ is generated by $\eps=\rho\circ\iota(\theta)$ and has the single defining relation
\[
\Sigma_1\Sigma_2\cdot\eps=0;
\]
equivalently
\begin{equation}
\label{eq:comm-part}
V\;\cong\;\R_{p,p-1}/(\Sigma_1\Sigma_2)\;\cong\;\prod_{\zeta\in(\F_p^\times)^2\setminus\{(1,1)\}}\F_p .
\end{equation}
The submodule $W\subseteq V$ is generated by $\rho\circ\iota(g\cdot\theta)$.
\end{lem}
\begin{proof}
Since $N_{p-1}^{\ab}$ is an $\R_{p-1}$-module, the quotient
$M=N_{p-1}^{\ab}/pN_{p-1}^{\ab}$ is a module over $\R_{p,p-1}$.  Projecting
the cyclic submodules of
Lemma~\ref{lem:ppminus1-mod} gives $V=\rho(\ab_{N_{p-1}}([F,F]))$, cyclic on
$\eps=\rho\circ\iota(\theta)$, and $W=\rho(\ab_{N_{p-1}}(C_H))$, cyclic on
$\rho\circ\iota(g\cdot\theta)$.

It remains to identify $V$ and to check that $\Sigma_1\Sigma_2$ is the single
defining relation. In $[F,F]^{\ab}$ the following identities hold for $n\ge1$:
\[
\overline{[u^n,v]}=\Bigl(\sum_{r=0}^{n-1}\kappa_F(\ab_F(u))^r\Bigr)\cdot\overline{[u,v]},\qquad
\overline{[u,v^n]}=\Bigl(\sum_{r=0}^{n-1}\kappa_F(\ab_F(v))^r\Bigr)\cdot\overline{[u,v]}.
\]
The identity $[ab,c]=[a,c]^b[b,c]$ gives
$[u^n,v]=[u^{n-1},v]^u\,[u,v]$.  Writing $\eta_n=\overline{[u^n,v]}$, we obtain
the recursion $\eta_n=\kappa_F(\ab_F(u))\cdot\eta_{n-1}+\overline{[u,v]}$,
$\eta_1=\overline{[u,v]}$, which recovers the first identity displayed above.
The identity for
$\overline{[u,v^n]}$ follows by symmetry via $[a,bc]=[a,c][a,b]^c$.  Taking $u=x_1$, $v=x_2$,
$n=p-1$, we obtain
\[
\overline{[x_1^{p-1},x_2^{p-1}]}
=\sum_{i=0}^{p-2}\sum_{j=0}^{p-2}t_1^i t_2^j\cdot\theta
=\Sigma_1\Sigma_2\cdot\theta .
\]
As $[x_1^{p-1},x_2^{p-1}]\in[N_{p-1},N_{p-1}]$, its class
$\ab_{N_{p-1}}([x_1^{p-1},x_2^{p-1}])$ vanishes, and by the $\R$-linearity of $\iota$,
\[
\Sigma_1\Sigma_2\cdot\eps
=\rho\bigl(\iota(\Sigma_1\Sigma_2\cdot\theta)\bigr)=\rho\bigl(\ab_{N_{p-1}}([x_1^{p-1},x_2^{p-1}])\bigr)=0.
\]
So $\beta:\R_{p,p-1}\to V$, $r\mapsto r\cdot\eps$, is a surjection with $J=(\Sigma_1\Sigma_2)\subseteq\ker\beta$.

We show $\ker\beta=J$ by comparing dimensions over $\F_p$.  By
Corollary~\ref{cor:JSigma},
\[
\R_{p,p-1}/J\;\cong\prod_{\zeta\in(\F_p^\times)^2\setminus\{(1,1)\}}\F_p ,
\qquad \dim_{\F_p}\R_{p,p-1}/J=(p-1)^2-1 .
\]

On the other hand $N_{p-1}$ has index $(p-1)^2$ in $F$, so  it
is free of rank $(p-1)^2+1$, whence $\dim_{\F_p}M=(p-1)^2+1$.  By
Lemma~\ref{lem:decompAb}\ref{dab:N},
$M=\rho(\ab_{N_{p-1}}(\langle h_1^{p-1},h_2^{p-1}\rangle))\oplus V$, and the first
summand is two-dimensional, so
\[
\dim_{\F_p}V=(p-1)^2-1=\dim_{\F_p}\R_{p,p-1}/J .
\]
As $J\subseteq\ker\beta$, the induced map
$\R_{p,p-1}/J\to V$ is an isomorphism, giving
$V\cong\R_{p,p-1}/(\Sigma_1\Sigma_2)$.
\end{proof}

\begin{thm}
\label{thm:p-part}
Let $H\le F$ be free of rank two with $A_H=F^{\ab}$, with $I_H=(g)$ as in
Proposition~\ref{prop:principal}.  Then $S=H\cap N_{p-1}$ is $\Gp$-dense in
$N_{p-1}$ if and only if $g$ has no zero in $(\F_p^\times)^2$.
\end{thm}

\begin{proof}
By \eqref{eq:WV} and Lemma~\ref{lem:pdense}, $S$ is $\Gp$-dense in $N_{p-1}$
(equivalently $\Cl_{\Gp}(S,N_{p-1})=N_{p-1}$) if and only if $W=V$.

Let $\eps=\rho\circ\iota(\theta)$, and let $\omega=\rho\circ\iota(g\cdot\theta)=g\cdot \eps$ be the generator of $W$ furnished by
Lemma~\ref{lem:VW}.  Under the isomorphism \eqref{eq:comm-part} of that
lemma,
\[
V\;\cong\;\R_{p,p-1}/(\Sigma_1\Sigma_2)\;\cong\;\prod_{\zeta\in(\F_p^\times)^2,\ \zeta\neq(1,1)}\F_p ,
\qquad
r\cdot\eps\;\longmapsto\;\bigl(r(\zeta)\bigr)_{\zeta},
\]
the cyclic $\R_{p,p-1}$-submodule $W$ generated by $\omega$ maps onto
\[
\bigl\{\bigl(r(\zeta)\,g(\zeta)\bigr)_{\zeta}:\ r\in\R_{p,p-1}\bigr\},
\]
the entries being multiplied one by one because each $\mathrm{ev}_\zeta$ is a
ring homomorphism.  Since by Lemma~\ref{lem:crt} the values $\bigl(r(\zeta)\bigr)_{\zeta}$, $r$
ranging over $\R_{p,p-1}$, already exhaust $\prod_{\zeta\ne(1,1)}\F_p$, the
products $r(\zeta)g(\zeta)$ range independently over the line $\F_p g(\zeta)$ in
each coordinate $\zeta$.  Hence the image is the subspace
$\prod_{\zeta\ne(1,1)}\F_p\,g(\zeta)$, of dimension
$\#\{\zeta\neq(1,1):\ g(\zeta)\neq0\}$.  Therefore
\[
W=V\ \Longleftrightarrow\ g(\zeta)\neq0\quad\text{for every }
\zeta\in(\F_p^\times)^2,\ \zeta\neq(1,1).
\]
Since $g(1,1)=1\neq0$ by Proposition~\ref{prop:principal}, requiring
$g(\zeta)\neq0$ only at $\zeta\neq(1,1)$ already means that $g$ has no zero in
$(\F_p^\times)^2$.  Together with the first equivalence this proves the theorem.
\end{proof}

\begin{thm}
\label{thm:main}
Let $H\le F$ be free of rank two with $A_H=F^{\ab}$, with $I_H=(g)$ as in
Proposition~\ref{prop:principal}.  Then
\[
\mathfrak D(H)=\PP\setminus\PP(g),
\]
where $\PP$ is the set of all primes and $\PP(g)=\{p:\ \exists\,\zeta\in(\F_p^\times)^2,\ g(\zeta)=0\}$.
\end{thm}

\begin{proof}
By Definition~\ref{defn:D}, $p\in\mathfrak D(H)$ if and only if $H$ is
$\Hp$-dense in $F$.  By Lemma~\ref{lem:decomp}, $H$ is $\Hp$-dense in $F$ if and
only if $H$ is $\Abp$-dense in $F$ and $S=H\cap N_{p-1}$ is $\Gp$-dense in
$N_{p-1}$.  Since $A_H=\ab_F(H)=F^{\ab}$, the image of $H$ in every
exponent-$(p-1)$ abelian quotient of $F$ (which factors through $F^{\ab}$) is
all of it, so $H$ is $\Abp$-dense in $F$ for every prime $p$.  By
Theorem~\ref{thm:p-part}, $S$ is then $\Gp$-dense in $N_{p-1}$ if and only if $g$
has no zero in $(\F_p^\times)^2$.  Hence
\[
p\in\mathfrak D(H)\ \Longleftrightarrow\ g\ \text{has no zero in }(\F_p^\times)^2
\ \Longleftrightarrow\ p\notin\PP(g),
\]
so $\mathfrak D(H)=\PP\setminus\PP(g)$.
\end{proof}

The theorem of the next section computes $\PP(g)$
 from \emph{finite} data.  Combining Theorem~\ref{thm:main} with
Proposition~\ref{prop:realize}, which
says that $g$ ranges over all of $\{g\in\R: g(1,1)=1\}$, identifies the
possible $\mathfrak D(H)$ exactly:
\begin{cor}
\label{cor:which-sets}
The sets of primes arising as $\mathfrak D(H)$, for $H\le F$ free of rank
two with $A_H=F^{\ab}$, are precisely the sets
$\PP\setminus\PP(g)$ with $g\in\R$ and $g(1,1)=1$.
\end{cor}

\begin{ex}
Take $g=t_1^2-t_1+1$, so that $g-1=(1-t_1)(-t_1)$ and
Proposition~\ref{prop:realize} applies with $s_2=-t_1$, $s_1=0$,
$v_1=1$, $v_2=\bigl([x_1,x_2]^{x_1}\bigr)^{-1}=[x_2,x_1]^{x_1}$:
\[
H=\bigl\langle x_1,\ x_2\,[x_2,x_1]^{x_1}\bigr\rangle,
\qquad I_H=(t_1^2-t_1+1).
\]
A zero of $g$ on $(\F_p^\times)^2$ is a root $\zeta_1\in\F_p$ of
$x^2-x+1$, necessarily nonzero since the constant term is $1$, together with
an arbitrary $\zeta_2\in \F_p^\times$.  For odd $p$ such a root exists if and only if the
discriminant $-3$ is a square in $\F_p$, that is if and only if $p=3$ or
$p\equiv1\bmod3$; for $p=2$ the polynomial reduces to $t_1^2+t_1+1$, which has
no root in $\F_2$.  Hence, by Theorem~\ref{thm:main},
\[
\mathfrak D(H)=\{p:\ p\equiv2\bmod 3\}.
\]
\end{ex}

\section{Deciding \texorpdfstring{$\mathfrak D(H)$}{D(H)}}
\label{sec:dec}
In this section we compute
\[
\PP(g)=\{p:\ \exists\,\zeta\in(\F_p^\times)^2,\ g(\zeta)=0\}
\]
for a fixed element $g\in\R=\Z[t_1^{\pm1},t_2^{\pm1}]$ with
$g(1,1)=1$.

Since a monomial vanishes nowhere on $(\F_p^\times)^2$ and takes the
value $1$ at $(1,1)$, the set $\PP(g)$ depends only on the ideal $(g)$,
not on its representative.  Replacing $g$ by a normalized generator (as in
Remark~\ref{rem:g-depends-on-basis}) changes neither $\PP(g)$ nor the
assumption $g(1,1)=1$, so we may assume that
\begin{equation}
\label{eq:norm-g}
g\in \Z[t_1,t_2],\quad t_1\nmid g,\ t_2\nmid g,\quad g(1,1)=1 .
\end{equation}
For the rest of this section $g$ denotes such a normalized polynomial.  If
$g$ is constant then the normalization forces $g=1$, so $\PP(g)=\emptyset$; we assume from now on that
$D=\deg g\ge1$.

\begin{thm}
\label{thm:torus}
Given a normalized $g$, there is an algorithm producing a finite Galois extension $L/\Q$, a
conjugacy-stable subset $A_g\subseteq G=\Gal(L/\Q)$, and a finite,
explicitly enumerated set $\mathcal S(g)$ of primes, such that every $p\notin \mathcal S(g)$ is
unramified in $L/\Q$ and satisfies
\begin{equation}
p\in\PP(g)\quad\Longleftrightarrow\quad
\Frob_{\mathfrak p}\in A_g\ \text{ for some (equivalently every) prime
$\mathfrak p$ of $\cO_L$ above $p$.}
\label{eq:star}
\end{equation}
\end{thm}

Here $\cO_L$ is the ring of integers of $L$.  For a prime $\mathfrak p$ of
$\cO_L$ above $p$, the Frobenius $\Frob_{\mathfrak p}\in G$ is the unique
element with $\Frob_{\mathfrak p}(\alpha)\equiv\alpha^{p}\pmod{\mathfrak p}$
for all $\alpha\in\cO_L$; it is defined exactly when $\mathfrak p$ is
unramified, as it is whenever $p\notin \mathcal S(g)$.  The equivalence of
``for some'' with ``for every'' follows from the transitivity of $G$ on the
primes above $p$: $\Frob_{\tau\mathfrak p}=\tau\Frob_{\mathfrak p}\tau^{-1}$
for $\tau\in G$, and $A_g$ is conjugacy-stable.

The rest of this section proves the theorem.  Subsection~\ref{sec:fix}
constructs the field $L$ and the ring $R=\cO_L[1/N]$ and records in
Lemma~\ref{lem:model} the properties of that model used later;
Subsection~\ref{sec:count} turns them into finite-field estimates; and
Subsection~\ref{sec:assem} assembles the theorem.

\subsection{Fixing the model}
\label{sec:fix}

We fix the convention that a
zero set is taken \emph{geometrically}: over an algebraic closure
of the field of coefficients.  Thus if $\mathcal{K}$ is a field with a chosen
algebraic closure $\overline {\mathcal{K}}$ and $T\subseteq \mathcal{K}[t_1,\dots,t_r]$ is any set of
polynomials, we write  $Z(T)$ for the zero set of $T$ in the affine space $\mathbb{A}^r(\overline {\mathcal{K}})$  over $\overline {\mathcal{K}}$, namely
\[
Z(T)=\bigl\{z\in\overline {\mathcal{K}}^{\,r}:\ f(z)=0\ \text{for all }f\in T\bigr\}
\subseteq\mathbb A^r(\overline {\mathcal{K}}),
\]
abbreviating $Z(f_1,\dots,f_s)=Z(\{f_1,\dots,f_s\})$; this set depends
only on the ideal generated by $T$, so we write $Z(\mathfrak a)$ for an
ideal $\mathfrak a$ as well.  For a subfield $\mathcal{K}_0\subseteq\overline {\mathcal{K}}$ and a
subset $X\subseteq\overline {\mathcal{K}}^{\,r}$ we write
\[
X(\mathcal{K}_0)=X\cap \mathcal{K}_0^{\,r}
\]
for the set of points of $X$ with all coordinates in $\mathcal{K}_0$, so that
$Z(T)=Z(T)(\overline {\mathcal{K}})$.  For the finite fields of characteristic $p$
occurring below we always take $\overline{\F_p}$ as the algebraic closure.

Throughout we work in a fixed algebraic closure $\overline\Q$: all data
are computed in $\overline\Q$, and every number that occurs denotes a
specific algebraic number.

Fix the graded lexicographic order on the monomials
$t_1^at_2^b$ (compare $a+b$ first, then $a$); the \emph{leading
coefficient} of a nonzero polynomial is the coefficient of its largest
monomial in this order.  A unit of $\overline\Q[t_1,t_2]$ being a nonzero
constant, two polynomials with leading coefficient $1$ are associate only
if they are equal.

By \cite{Duval91}, one can
compute a \emph{normalized absolute} factorization
\begin{equation}
g=c\prod_{i\in\mathcal I}h_i^{n_i}
\label{ap:fact}
\end{equation}
whose coefficients lie in a computable subfield of $\overline\Q$, where
the $h_i$ are pairwise nonassociate,
absolutely irreducible, with leading coefficient $1$, and
$c\in\overline\Q^\times$.  Since $t_1,t_2\nmid g$, no $h_i$ is a
monomial.  Write $Z_i=Z(h_i)$ for the absolute irreducible components of
$Z(g)$, and let $\delta_i=\deg h_i$.

For $i<j$ the polynomials $h_i$ and $h_j$ are nonassociate irreducibles of
the unique factorization domain $\overline\Q[t_1,t_2]$, hence coprime, so
the plane curves $Z_i$ and $Z_j$ share no component and meet in at most
$\delta_i\delta_j$ points by B\'ezout's theorem \cite{Fulton}, each with algebraic
coordinates.  Let
\[
\Omega_{ij}=Z(h_i,h_j)\cap(\overline\Q^\times)^2
\]
be the points among them in the torus $(\overline\Q^\times)^2$, and let
\[
\Omega=\bigcup_{i<j}\Omega_{ij}
=\bigl\{z\in(\overline\Q^\times)^2:\ h_i(z)=h_j(z)=0
\ \text{for some }i\ne j\bigr\}
\]
be the finite set of multiple-component torus points of $Z(g)$.  Let
\[
U_i=\bigl(Z_i\cap(\overline\Q^\times)^2\bigr)\setminus\bigcup_{j\ne i}Z_j .
\]
Every torus point of $Z(g)$ lies on some $Z_i$ and, if on several,
belongs to $\Omega$; hence disjointly
\begin{equation}
Z(g)\cap(\overline\Q^\times)^2
=\Bigl(\bigsqcup_{i\in\mathcal I}U_i\Bigr)\sqcup\Omega .
\label{ap:decomp}
\end{equation}

\subsubsection{The field \texorpdfstring{$L$}{L}.}\label{sec:l}  Let
\[
L=\Bigl(\Q\bigl(\{\text{coefficients of the }h_i\}\cup\{c\}\cup
\{\text{coordinates of the points of }\Omega\}\bigr)
\Bigr)^{\mathrm{gal}} .
\]
be the Galois closure over $\Q$ of the field generated by the coefficients
of the $h_i$, the constant $c$, and the coordinates of the points of
$\Omega$.  As that field is finite over $\Q$, $L/\Q$ is a finite Galois
extension.  Moreover, we  have $h_i\in L[t_1,t_2]$, $c\in L^\times$, and
$\Omega\subset(L^\times)^2$.

Let $G=\Gal(L/\Q)$.  Every $\sigma\in G$ fixes the rational numbers, so
$\sigma(g)=g$.  Applied to \eqref{ap:fact}, $\sigma$ gives a second
absolute irreducible decomposition of the same polynomial, with factors
$\sigma(h_i)$ of the same degrees.  The multiset of absolutely irreducible
factors of $g$, with multiplicities, is unique; hence for some
$j\in\mathcal I$ the polynomial $\sigma(h_i)$ is associate to $h_j$.  As
both are monic, $\sigma(h_i)=h_j$.  Thus $\sigma$ permutes $\mathcal I$;
write $i\mapsto\sigma(i)$, so that $\sigma(h_i)=h_{\sigma(i)}$ and
$n_{\sigma(i)}=n_i$.  This makes $\mathcal I$ a $G$-set.

The set $\Omega$ is a $G$-set as well.  The coefficients of the $h_i$ and
the coordinates of the points of $\Omega$ lie in $L$, so every $\sigma\in
G$ acts on both; and as $\sigma$ is a ring homomorphism, evaluation commutes
with it,
\[
\bigl(\sigma h_i\bigr)(\sigma z)=\sigma\bigl(h_i(z)\bigr) .
\]
Let $z\in\Omega$, say $h_i(z)=h_j(z)=0$ with $i\ne j$.  Applying $\sigma$
then gives $h_{\sigma(i)}(\sigma z)=0$ and $h_{\sigma(j)}(\sigma z)=0$.
Here $\sigma(i)\ne\sigma(j)$, because $\sigma$ permutes $\mathcal I$; and
$\sigma$ carries $L^\times$ into $L^\times$, so $\sigma z$ is a torus
point.  Hence $\sigma z\in\Omega$, so $\sigma(\Omega)\subseteq\Omega$.
Replacing $\sigma$ by $\sigma^{-1}$ gives the reverse inclusion; hence
$\sigma(\Omega)=\Omega$.

\begin{rem}
\label{rem:omega-calc}
The points of $\Omega$ and the field $L$ are computable, everything being
a computation with algebraic numbers.  The absolute
factorization \eqref{ap:fact} supplies the field
$L_0=\Q(\{\text{coefficients of the }h_i\}\cup\{c\})$.  For $i<j$ the
ideal $(h_i,h_j)\subset L_0[t_1,t_2]$ is zero-dimensional, so a Gr\"obner
computation \cite{GG} returns the points of $Z_i\cap Z_j$; deleting the
axial ones gives $\Omega_{ij}$ and $\Omega$.  The field $L$ of
Subsection~\ref{sec:l} is the Galois closure over $\Q$ of
$L_0\bigl(\{\text{coordinates of the points of }\Omega\}\bigr)$, and the
algorithms of \cite{Cohen} compute a primitive element $\beta$ with
$L=\Q(\beta)$ and the Galois group $G=\Gal(L/\Q)$.  Expressing the
coefficients of the $h_i$ and the coordinates of $\Omega$ in the power
basis $1,\beta,\dots,\beta^{[L:\Q]-1}$ of $L$, every $\sigma\in G$ acts
on $\mathcal I$ ($\sigma(h_i)=h_{\sigma(i)}$) and on $\Omega$ by explicit
permutations.
\end{rem}

\subsubsection{The ring \texorpdfstring{$R=\cO_L[1/N]$}{R=OL[1/N]}}

For $\alpha\in L^\times$ write
$\alpha=a_\alpha/b_\alpha$ with
$a_\alpha,b_\alpha\in\cO_L\setminus\{0\}$ and let
\[
\nu(\alpha)=\bigl|N_{L/\Q}(a_\alpha)\,N_{L/\Q}(b_\alpha)\bigr|
\in\Z_{\ge1} ;
\]
Inverting $\nu(\alpha)$ makes $\alpha$ a \emph{unit}. For $a\in\cO_L\setminus\{0\}$ one has
$N_{L/\Q}(a)=\prod_{\sigma\in\Gal(L/\Q)}\sigma(a)$ because $L/\Q$ is Galois,
whence $N_{L/\Q}(a)/a=\prod_{\sigma\ne1}\sigma(a)\in\cO_L$ and $a$ divides
$N_{L/\Q}(a)$ in $\cO_L$.  Hence in $\cO_L[1/\nu(\alpha)]$ both
$a_\alpha$ and $b_\alpha$ divide a unit, so $\alpha$ and
$\alpha^{-1}$ are there both integral.  For a finite subset
$\mathcal A\subset L^\times$ we abbreviate
$\nu(\mathcal A)=\prod_{\alpha\in\mathcal A}\nu(\alpha)$. Define
\[
N=\bigl|\mathrm{disc}(L/\Q)\bigr|\cdot\prod_{1\leq i\leq 5}\nu(\mathcal{A}_i) ,
\]
where $\mathrm{disc}(L/\Q)$ is the discriminant of $L$ over $\Q$ and the
five finite subsets $\mathcal A_i\subset L^\times$ are the following:
\begin{enumerate}[label=$\mathcal{A}_{\arabic*}$, ref=$\mathcal{A}_{\arabic*}$]
\item $=\{\text{nonzero coefficients of the
}h_i\}\cup\{c\}$, the data of the factorization \eqref{ap:fact};
\item $=\{\gamma_{ij}\}_{i\ne j}$, where $\gamma_{ij}$ is
a nonzero coefficient of $h_i-h_j$ --- there is one, the $h_i$ being
pairwise distinct;
\item\label{A3} $=\{a_z,b_z\}_{z\in\Omega}\;\cup\;\bigl(\{a_z-a_w,\
b_z-b_w:\ z\ne w\ \text{in }\Omega\}\setminus\{0\}\bigr)$, the
coordinates of the points of $\Omega$ together with the nonzero
coordinate-differences of distinct points;
\item $=\bigcup_{i\in\mathcal I}\mathcal A(h_i)$, the sets furnished by
Proposition~\ref{prop:red-irr} of Appendix~\ref{app:ruppert}; they keep
each $\bar h_i$ absolutely irreducible;
\item $=$ the set of nonzero coefficients of the cofactor
polynomials of the certificates produced in the paragraph
\emph{Certificates at the multiple-component points} below.
\end{enumerate}

\emph{Certificates at the multiple-component points.}
We carry $\Omega_{ij}$ to characteristic $p$ by means of an ideal.  The
ideal $(h_i,h_j)$ is unsuitable: its zero set includes the points of
$Z_i\cap Z_j$ on the axes, which are excluded from $\Omega_{ij}$.  Therefore we adjoin a variable $u$ subject to $t_1t_2u=1$ and let,
for $i<j$,
\[
\mathfrak L_{ij}=(h_i,\ h_j,\ t_1t_2u-1)\subset L[t_1,t_2,u].
\]
 Here $h_i$ and $h_j$ are read as polynomials in $L[t_1,t_2,u]$ through
the inclusion $L[t_1,t_2]\subset L[t_1,t_2,u]$, so they do not involve
$u$; in particular their support, bidegree and total degree are those
computed in $t_1,t_2$.  The symbol $Z(h_i,h_j)$ denotes the  zero set
in $\mathbb A^2(\overline{\Q})$, never the zero set in
$\mathbb A^3(\overline{\Q})$; zero sets in $\mathbb A^3(\overline{\Q})$ occur
only for $\mathfrak L_{ij}$, $\mathcal{J}_{ij}$, and the $\mathfrak m_z$
below.  Write
\[
\pr\colon\mathbb A^3(\overline{\Q})\longrightarrow\mathbb A^2(\overline{\Q}),\qquad
(t_1,t_2,u)\longmapsto(t_1,t_2) .
\]
Since $t_1t_2u=1$ forces $t_1,t_2\ne0$ and then determines $u$, the map
$\pr$ restricts to a bijection from $Z(\mathfrak L_{ij})$ onto $\Omega_{ij}$.  Thus
$\mathfrak L_{ij}$ presents $\Omega_{ij}$ as a closed subset of affine
$3$-space, accessible to ideal-theoretic computation and, in
Lemma~\ref{lem:model} below, to reduction modulo primes.  As $\Omega_{ij}\subset(L^\times)^2$ by the construction of $L$,
the points of $Z(\mathfrak L_{ij})$ have their coordinates in $L$ as well, that is
$Z(\mathfrak L_{ij})=Z(\mathfrak L_{ij})(L)$.  Writing $z=(a_z,b_z)$ for a point of
$\Omega_{ij}$,
so that $\bigl(a_z,b_z,(a_zb_z)^{-1}\bigr)$ is the point of $Z(\mathfrak L_{ij})$
above it, let
\[
\mathfrak m_z=\bigl(t_1-a_z,\ t_2-b_z,\ u-(a_zb_z)^{-1}\bigr)
\subset L[t_1,t_2,u].
\qquad
\mathcal{J}_{ij}=\prod_{z\in\Omega_{ij}}\mathfrak m_z ,
\]
 with the convention $\mathcal{J}_{ij}=(1)$ when $\Omega_{ij}=\emptyset$.
 Each $\mathfrak m_z$ is a maximal ideal
of $L[t_1,t_2,u]$ with $Z(\mathfrak m_z)$ the single point
$\bigl(a_z,b_z,(a_zb_z)^{-1}\bigr)$.

  Write
$\mathcal{J}_{ij}=(\varphi_1,\dots,\varphi_s)$, where the $\varphi_l$ are the
polynomials $\prod_{z\in\Omega_{ij}}g_z$ expanded explicitly, one $g_z$
chosen from the three generators
of each $\mathfrak m_z$.  The next lemma identifies $\mathcal{J}_{ij}$ with
the radical of $\mathfrak L_{ij}$; its proof is deferred to
Appendix~\ref{app:radical}.

\begin{lem}
\label{lem:radical}
One has
\[
\mathcal{J}_{ij}=\bigcap_{z\in\Omega_{ij}}\mathfrak m_z=\sqrt{\mathfrak L_{ij}}.
\]
\end{lem}

Since $\mathcal{J}_{ij}=\sqrt{\mathfrak L_{ij}}$ and $L[t_1,t_2,u]$ is
Noetherian, each generator $\varphi_l$ has a power
$\varphi_l^{c_l}\in\mathfrak L_{ij}$; with $\mu_{ij}=\sum_l(c_l-1)+1$ a finite
sum of the $c_l$ (each found by Gr\"obner reduction \cite{GG}), any
product of $\mu_{ij}$ of the $\varphi_l$ repeats some $\varphi_l$ at least
$c_l$ times and so lies in $\mathfrak L_{ij}$.  Together with
$\mathfrak L_{ij}\subseteq\sqrt{\mathfrak L_{ij}}=\mathcal{J}_{ij}$ this gives
the inclusion-chain
\begin{equation}
\mathcal{J}_{ij}^{\mu_{ij}}\subseteq \mathfrak L_{ij}\subseteq \mathcal{J}_{ij} .
\label{ap:chain}
\end{equation}

Ideal-membership computations \cite{GG} produce explicit
\emph{certificates} for the two inclusions of \eqref{ap:chain}: not just
the inclusions themselves, but the cofactor polynomials witnessing them,
which survive reduction.  Explicitly, writing
$f_1,f_2,f_3=h_i,h_j,t_1t_2u-1$ for the generators of $\mathfrak L_{ij}$, a
certificate for $\mathfrak L_{ij}\subseteq \mathcal{J}_{ij}$ is a family
$q_{kl}\in L[t_1,t_2,u]$ with
\[
f_k=\sum_{l=1}^{s}q_{kl}\,\varphi_l\qquad(k=1,2,3),
\]
and a certificate for $\mathcal{J}_{ij}^{\mu_{ij}}\subseteq \mathfrak L_{ij}$ is, for each of
the finitely many generators
$\psi=\varphi_{l_1}\cdots\varphi_{l_{\mu_{ij}}}$ of
$\mathcal{J}_{ij}^{\mu_{ij}}$, a triple $q^\psi_1,q^\psi_2,q^\psi_3$ with
\[
\psi=q^\psi_1h_i+q^\psi_2h_j+q^\psi_3(t_1t_2u-1) .
\]
If $\Omega_{ij}=\emptyset$ then $\mathcal{J}_{ij}=(1)$ and the second family
degenerates to the single identity
$1=q_1h_i+q_2h_j+q_3(t_1t_2u-1)$, a certificate for $1\in \mathfrak L_{ij}$.
Let $\mathcal{A}_5\subset L^\times$ be the set of all nonzero coefficients of all the cofactor polynomials above, over all pairs $i<j$.

\begin{rem}
\label{rem:Acompute}
Each $\mathcal A_i$ is computable, so $N$ is an explicit integer.
Concretely, $\mathcal A_1$ is the data of the absolute factorization
\eqref{ap:fact}; $\mathcal A_2$ is read off from the
pairwise differences $h_i-h_j$; $\mathcal A_3$ comes from $\Omega$ via
its coordinates and their differences, as in Remark~\ref{rem:omega-calc};
the sets $\mathcal A_4$ are furnished explicitly by
Proposition~\ref{prop:red-irr} of Appendix~\ref{app:ruppert};
and $\mathcal A_5$ is the set of nonzero coefficients of the certificate
cofactors produced by the ideal-membership computations above.
\end{rem}

Now let $R=\cO_L[1/N]$.
By construction every element of
$\mathcal A_1\cup\cdots\cup\mathcal A_5$ is a unit of $R$. The nonzero
primes of $R$ are the ideals $\mathfrak pR$ with $\mathfrak p$ a prime of
$\cO_L$ not dividing $N$.  Fix such a $\mathfrak p$ and let $p$ be the
rational prime it lies over.
Since $\cO_L$ is a Dedekind domain, the nonzero prime $\mathfrak p$ is
maximal, so $\cO_L/\mathfrak p$ is a field; and since $N\notin\mathfrak p$
the image $\bar{N}$ of $N$ is invertible there, whence
\[
\F_{\mathfrak p}=R/\mathfrak pR=(\cO_L/\mathfrak p)[1/\bar N]
=\cO_L/\mathfrak p ,
\]
so that reduction from $R$ and reduction from $\cO_L$ have the same target
and we may use either.  Note that $\mathfrak p\cap\Z$ is a prime ideal of
$\Z$, and it is nonzero because $N_{L/\Q}(\alpha)\in\mathfrak p\cap\Z
\setminus\{0\}$ for any $\alpha\in\mathfrak p\setminus\{0\}$; by definition
of $p$ it is $p\Z$.  Hence the composite $\Z\to\cO_L\to\F_{\mathfrak p}$
has kernel exactly $p\Z$ and therefore induces an embedding
$\F_p\hookrightarrow\F_{\mathfrak p}$.  Since $\cO_L$ is generated
by $[L:\Q]$ elements as a $\Z$-module, $\F_{\mathfrak p}$ is an
$\F_p$-vector space of some dimension $m\le[L:\Q]$; thus
$\F_{\mathfrak p}=\F_{p^m}$ is a finite field of characteristic $p$, and $\overline{\F_{\mathfrak p}}=\overline{\F_p}$.

The quotient map $\pi_{\mathfrak p}:R\to\F_{\mathfrak p}$
extends uniquely to a ring homomorphism
\[
(\pi_{\mathfrak p})_*:R[t_1,t_2]\longrightarrow\F_{\mathfrak p}[t_1,t_2],
\qquad
\sum_\alpha c_\alpha t^\alpha\longmapsto
\sum_\alpha\pi_{\mathfrak p}(c_\alpha)t^\alpha ,
\]
acting as the identity on $t_1,t_2$, and likewise on $R[t_1,t_2,u]$.
Thus, in what follows, $\bar\alpha=\pi_{\mathfrak p}(\alpha)$ for $\alpha\in R$,
$\bar f=(\pi_{\mathfrak p})_*(f)$ for a polynomial $f$ in $R[t_1,t_2]$ or
$R[t_1,t_2,u]$, and $\bar z$ denotes the coordinatewise reduction of a
point $z$.

\subsubsection{The model lemma}
We summarize the outcome in a single lemma, the model on which
Subsection~\ref{sec:count} operates.

\begin{lem}
\label{lem:model}
Given a normalized $g$ with factorization \eqref{ap:fact}, one can compute a finite Galois extension $L/\Q$
and an integer $N$, such that, with $R=\cO_L[1/N]$ and
$\bar Z_i=Z(\bar h_i)\subset\mathbb A^2(\overline{\F_p})$, the following hold for every prime $\mathfrak p$
of $\cO_L$ not dividing $N$.
\begin{enumerate}[label=(\alph*), ref=(\alph*)]
\item\label{mod:fact} Each $\bar h_i\in\F_{\mathfrak p}[t_1,t_2]$ is
absolutely irreducible of the same total degree $\delta_i=\deg h_i$ as
$h_i$, the $\bar h_i$ are pairwise nonassociate, and
$Z(\bar g)=\bigcup_{i\in\mathcal I}\bar Z_i$.
\item\label{mod:bound} Geometrically, $\bar Z_i$ has at most $2\delta_i$
points on the two axes, its projective closure at most $\delta_i$ points
at infinity, and $\bar Z_i\cap\bar Z_j$ at most $\delta_i\delta_j$ points
for $j\ne i$.
\item\label{mod:decomp} Reduction is injective on $\Omega$ and, with
$\bar\Omega=\{\bar z:\ z\in\Omega\}\subset(\F_{\mathfrak p}^\times)^2$ and
\[
\bar U_i=\bigl(\bar Z_i\cap(\overline{\F_p}^\times)^2\bigr)
\setminus\bigcup_{j\ne i}\bar Z_j ,
\]
one has the disjoint decomposition
\begin{equation}
Z(\bar g)\cap(\overline{\F_p}^\times)^2
=\Bigl(\bigsqcup_{i\in\mathcal I}\bar U_i\Bigr)\sqcup\bar\Omega ,
\label{ap:decomp-bar}
\end{equation}
the exact analogue of \eqref{ap:decomp} after reduction.
\item\label{mod:frob} $\mathfrak p$ is unramified in $L/\Q$, so that
$\sigma=\Frob_{\mathfrak p}\in G$ is defined; $R$ is $G$-stable; and
$\overline{\sigma(\alpha)}=\bar\alpha^{\,p}$ for every $\alpha\in R$,
whence $\bar\alpha\in\F_p$ whenever $\sigma(\alpha)=\alpha$.
\end{enumerate}
\end{lem}

\begin{proof}
Take $L$ and $N$ as constructed above.  Fix a prime $\mathfrak p$ of $\cO_L$ not dividing $N$.
Recall that every element of $\mathcal A_1\cup\dots\cup\mathcal A_5$ is a
unit of $R$.

\ref{mod:fact}.  The coefficients of the $h_i$, and $c$, are units of
$R$, so $\bar h_i$ retains the support of $h_i$, hence its bidegree, its
total degree $\delta_i$, and its leading coefficient $1$.  Reducing
\eqref{ap:fact} gives $\bar g=\bar c\prod_i\bar h_i^{n_i}$ with
$\bar c\in\F_{\mathfrak p}^\times$, whence $Z(\bar g)=\bigcup_i\bar Z_i$.

The nonzero coefficients of $h_i$ and the
elements of $\mathcal A(h_i)$ are units of $R$, so
Proposition~\ref{prop:red-irr} with $E=L$, $A=R$ and
$\mathfrak q=\mathfrak pR$ shows $\bar h_i$ absolutely irreducible.  Finally
$\bar\gamma_{ij}\ne0$ gives $\bar h_i\ne\bar h_j$ for $i\ne j$; and being
monic, two such polynomials are associate only if they are equal.

\ref{mod:bound}.  Since $\bar h_i$ has the support of $h_i$ and
$t_1,t_2\nmid h_i$, the restrictions $\bar h_i(0,t_2)$ and
$\bar h_i(t_1,0)$ are nonzero of degree at most $\delta_i$, so $\bar Z_i$
meets each axis in at most $\delta_i$ points.  Its top-degree form is a
nonzero form of degree $\delta_i$ in two variables, so its projective
closure has at most $\delta_i$ points at infinity.  By \ref{mod:fact} the
curves $\bar Z_i,\bar Z_j$ for $i\ne j$ share no component, hence meet
in at most $\delta_i\delta_j$ points, by the plane B\'ezout theorem.

\ref{mod:decomp}.  Each coordinate $a_z,b_z$ is a unit of $R$, so $z$
reduces to $\bar z\in(\F_{\mathfrak p}^\times)^2$.  If $z\ne w$, one of
$a_z-a_w$, $b_z-b_w$ is a nonzero unit, hence reduces to a nonzero element,
so $\bar z\ne\bar w$; reduction is injective on $\Omega$.

We claim next that, for $i<j$, the torus points lying on both $\bar Z_i$ and
$\bar Z_j$ are exactly the $\bar z$ with $z\in\Omega_{ij}$.  The generators
of the $\mathfrak m_z$ and the cofactors of the two certificates lie in
$R$, so the identity $\mathcal{J}_{ij}=\prod_z\mathfrak m_z$ and both
inclusions of \eqref{ap:chain} also hold in $R[t_1,t_2,u]$; reducing modulo
$\mathfrak p$ gives
\[
\bar{\mathcal{J}}_{ij}^{\,\mu_{ij}}\subseteq\bar{\mathfrak L}_{ij}\subseteq
\bar{\mathcal{J}}_{ij}
\qquad\text{in }\F_{\mathfrak p}[t_1,t_2,u] ,
\]
taken, per our convention, in $\mathbb A^3(\overline{\F_p})$.  Since
$Z(\mathfrak a^\mu)=Z(\mathfrak a)$, the chain forces
$Z(\bar{\mathfrak L}_{ij})=Z(\bar{\mathcal{J}}_{ij})$; and as reduction carries a
product of ideals to the product of the images,
$\bar{\mathcal{J}}_{ij}=\prod_z\bar{\mathfrak m}_z$, so
\[
Z(\bar{\mathfrak L}_{ij})=Z(\bar{\mathcal{J}}_{ij})
=\bigcup_{z\in\Omega_{ij}}Z(\bar{\mathfrak m}_z)
=\bigl\{\bigl(\bar a_z,\bar b_z,(\bar a_z\bar b_z)^{-1}\bigr):
\ z\in\Omega_{ij}\bigr\} ,
\]
each $Z(\bar{\mathfrak m}_z)$ being a single point.  Now $Z(\bar{\mathfrak L}_{ij})$
is exactly the torus points lying on both $\bar Z_i$ and $\bar Z_j$, each
tagged with $u=(t_1t_2)^{-1}$, so projecting the display to the $(t_1,t_2)$-plane yields the
claim.  If $\Omega_{ij}=\emptyset$, the certificate for
$1\in \mathfrak L_{ij}$
reduces to one for $1\in\bar{\mathfrak L}_{ij}$ and gives
$Z(\bar{\mathfrak L}_{ij})=\emptyset$, the same conclusion.

The decomposition now follows as over $\overline\Q$.  By \ref{mod:fact}
every torus point of $Z(\bar g)$ lies on some $\bar Z_i$; on exactly one, it
lies in the corresponding $\bar U_i$ and in no other, the $\bar U_j$ being
pairwise disjoint; on at least two, it lies in $\bar\Omega$ by the claim and
in no $\bar U_j$.  Conversely, each $\bar U_i$ and $\bar\Omega$ consists of
torus points of $Z(\bar g)$, giving \eqref{ap:decomp-bar}.

\ref{mod:frob}.  Since $N\in\Z$ and $\mathfrak p\cap\Z=p\Z$, the hypothesis
$N\notin\mathfrak p$ gives $p\nmid N$; as $\mathrm{disc}(L/\Q)$ divides $N$, $p$
is unramified in $L/\Q$ \cite{Neukirch} and $\sigma=\Frob_{\mathfrak p}$ is
defined.  Since $N$ is fixed by $G$, $\tau(R)=R$ for every $\tau\in G$.
Being in the decomposition group of $\mathfrak p$, $\sigma$ satisfies
$\sigma(\mathfrak pR)=\mathfrak pR$ and so induces on
$\F_{\mathfrak p}=R/\mathfrak pR=\cO_L/\mathfrak p$ the $p$-power map, the
Frobenius at the unramified prime; that is
$\overline{\sigma(\alpha)}=\bar\alpha^{\,p}$.  If $\sigma(\alpha)=\alpha$
this reads $\bar\alpha^{\,p}=\bar\alpha$, and the $p$-power map on
$\F_{\mathfrak p}$ fixes exactly $\F_p$.
\end{proof}

\subsection{Counting \texorpdfstring{$\F_p$}{Fp}-points}
\label{sec:count}

Throughout this subsection $p$ is a rational prime not dividing $N$ and
$\mathfrak p$ is a prime of $\cO_L$ above $p$.  By the convention of
\S\ref{sec:fix}, $\F_{\mathfrak p}$ is a subfield of the fixed algebraic
closure $\overline{\F_p}$ and all the reduced objects of
Lemma~\ref{lem:model} live in $(\overline{\F_p}^\times)^2$, the notation
$X(\F_p)=X\cap\F_p^2$ being available for any subset $X$ of the affine plane $\mathbb{A}^2(\overline{\F_p})$.  When
$X$ is a subset of the projective plane $\mathbb P^2(\overline{\F_p})$ we likewise write
$X(\F_p)$ for the points of $X$ admitting homogeneous coordinates in
$\F_p$.

\begin{lem}
\label{lem:weil}
Let $\bar h\in\F_p[t_1,t_2]$ be absolutely irreducible of total degree
$\delta$, and let $\bar Z^{+}\subseteq\mathbb P^2(\overline{\F_p})$ be the
projective closure of $Z(\bar h)$.  Then
\[
\bigl|\#\bar Z^{+}(\F_p)-(p+1)\bigr|\le(\delta-1)(\delta-2)\sqrt p .
\]
\end{lem}

\begin{proof}
The homogenization
$\bar h^{+}(T_0,T_1,T_2)=T_0^{\delta}\,\bar h(T_1/T_0,T_2/T_0)$ is a form
of degree $\delta$ with coefficients in $\F_p$, not divisible by $T_0$.
 It is again
absolutely irreducible: since $T_0\nmid\bar h^{+}$, every factor of
$\bar h^{+}$ has a monomial free of $T_0$, hence substituting $T_0=1$
sends each factor of $\bar h^{+}$ to a nonconstant polynomial; a
factorization of the homogeneous $\bar h^{+}$ would thus give one of
$\bar h$, contradicting that $\bar h$ is absolutely irreducible.  

Thus $\bar Z^{+}=Z(\bar h^{+})$ is a reduced, absolutely irreducible projective
plane curve of degree $\delta$ over $\F_p$. Its \emph{arithmetic} genus is $\tfrac{(\delta-1)(\delta-2)}2$, a number depending only on the degree.  The Weil bound
\cite[Corollary~2.4]{AP} gives, for every absolutely irreducible projective curve $X$
over a finite field $\F_p$,
\[
\bigl|\#X(\F_p)-(p+1)\bigr|\le 2\,\gamma_X\sqrt p,
\]
where $\gamma_X$ denotes the arithmetic genus of $X$.  Applying this
to $X=\bar Z^{+}$ gives the assertion.
\end{proof}

Recall $D=\deg g\ge1$; then
$\delta_i\le D$ for every $i$.  Let
\begin{equation}
\mathcal S(g)=\{p:\ p\mid N\}\cup\{p:\ p\le9D^4\}.
\label{ap:Sg}
\end{equation}

\begin{lem}
\label{lem:Ui}
Let $p\notin \mathcal S(g)$, $\mathfrak p$ a prime of $\cO_L$ over $p$, and $i\in\mathcal I$ with $\bar h_i\in\F_p[t_1,t_2]$. Then $\bar U_i(\F_p)\ne\emptyset$.
\end{lem}

\begin{proof}
By Lemma~\ref{lem:model}\ref{mod:fact}, $\bar h_i$ is absolutely
irreducible of degree $\delta_i$, and by hypothesis it lies in
$\F_p[t_1,t_2]$; thus Lemma~\ref{lem:weil}, applied to its projective
closure $\bar Z_i^{+}$, gives
\[
\#\bar Z_i^{+}(\F_p)\ \ge\ p+1-(\delta_i-1)(\delta_i-2)\sqrt p
\ \ge\ p+1-D^2\sqrt p ,
\]
since $(\delta_i-1)(\delta_i-2)\le D^2$.  Passing from $\bar Z_i^{+}$ to
$\bar U_i$ removes the points at infinity and on the axes, at most
$3\delta_i\le3D$ of them, and those on some $\bar Z_j$, $j\ne i$, at most
$\delta_i\sum_{j\ne i}\delta_j\le D^2$, both by
Lemma~\ref{lem:model}\ref{mod:bound}.  Hence
\[
\#\bar U_i(\F_p)\ \ge\ p+1-D^2\sqrt p-(3D+D^2)
\ >\ p-D^2\sqrt p-4D^2,
\]
as $3D\le3D^2$.  Since $p\notin \mathcal S(g)$, we have $\sqrt p>3D^2$, and
\[p-D^2\sqrt p=\sqrt p(\sqrt p-D^2)>3D^2\cdot 2D^2=6D^4\ >\ 4D^2 .\]
So $\#\bar U_i(\F_p)>0$, i.e.\ $\bar U_i(\F_p)\ne\emptyset$.
\end{proof}

\begin{lem}
\label{lem:count}
For every prime $p\notin \mathcal S(g)$ and every prime $\mathfrak p$ of
$\cO_L$ lying over $p$, let $\sigma=\Frob_{\mathfrak p}$. Then
\begin{enumerate}
\item $\Bigl(\bigcup_{i\in\mathcal I}\bar U_i\Bigr)(\F_p)\ne\emptyset
\Longleftrightarrow \sigma\text{ fixes some }i\in\mathcal I$;
\item $\bar\Omega(\F_p)\ne\emptyset
\Longleftrightarrow \sigma\text{ fixes some }z\in\Omega$.
\end{enumerate}
\end{lem}

\begin{proof}
 Let $\fr$ be the $p$-power map of $\overline{\F_p}$:
$a\mapsto a^{p}$. It acts componentwise on the coordinates of a point and on polynomial
coefficients.  By Subsection~\ref{sec:l},
$\sigma(h_i)=h_{\sigma(i)}$, and by
Lemma~\ref{lem:model}\ref{mod:frob}, $\overline{\sigma(\alpha)}
=\bar\alpha^{\,p}$ for $\alpha\in R$.  Hence
$\fr(\bar h_i)=\bar h_{\sigma(i)}$, so $\fr$ maps $\bar Z_i$ onto
$\bar Z_{\sigma(i)}$ and, since it preserves
$(\overline{\F_p}^\times)^2$, also $\bar U_i$ onto
$\bar U_{\sigma(i)}$, for every $i\in\mathcal I$.

(1) Pick $P$ in $\bigl(\bigcup_{i\in\mathcal I}\bar U_i\bigr)(\F_p)$.
Then $P\in\bar U_i$ for some $i$, and $\fr(P)=P$ because $P\in\F_p^2$.
As $P\in\bar U_i$, the point
$P=\fr(P)\in\fr(\bar U_i)=\bar U_{\sigma(i)}$, so
$P\in\bar U_i\cap\bar U_{\sigma(i)}$; the $\bar U_j$ being pairwise
disjoint forces $\sigma(i)=i$.

Conversely, let $\sigma(i)=i$.  Each coefficient $\alpha$ of $h_i$ is
fixed by $\sigma$ and lies in $\mathcal A_1\subset R^\times$, so by
Lemma~\ref{lem:model}\ref{mod:frob},
\[
\bar\alpha^{\,p}=\overline{\sigma(\alpha)}=\bar\alpha,
\]
and since the $p$-power map fixes exactly $\F_p$ inside
$\F_{\mathfrak p}$, every reduced coefficient of $\bar h_i$ lies in
$\F_p$, so $\bar h_i\in\F_p[t_1,t_2]$; Lemma~\ref{lem:Ui} yields
$\bar U_i(\F_p)\ne\emptyset$.

(2) By Lemma~\ref{lem:model}\ref{mod:decomp},
$z\mapsto\bar z$ is injective on $\Omega$, and as the coordinates of $z$
lie in $\mathcal A_3\subset R^\times$, Lemma~\ref{lem:model}\ref{mod:frob}
gives $\overline{\sigma(z)}=\fr(\bar z)$ for $z\in\Omega$.  A point of
$(\F_{\mathfrak p}^\times)^2$ is fixed by $\fr$ exactly when its
coordinates lie in $\F_p$; hence
\[
\bar z\in(\F_p^\times)^2
\iff\fr(\bar z)=\bar z
\iff\overline{\sigma(z)}=\bar z
\iff\sigma(z)=z,
\]
since $\sigma(z)\in\Omega$ ($\Omega$ is $G$-stable) and reduction is
injective on $\Omega$.  Thus reduction maps
$\{z\in\Omega:\ \sigma(z)=z\}$ bijectively onto $\bar\Omega(\F_p)$, and
$\bar\Omega(\F_p)\ne\emptyset$ if and only if $\sigma$ fixes a point of
$\Omega$.
\end{proof}

\subsection{Proof and Consequences}
\label{sec:assem}

\subsubsection*{Proof of
Theorem~\ref{thm:torus}}

Define the finite sets
\[
A_g=
\bigl\{\sigma\in G:\ \sigma(i)=i\ \text{for some }i\in\mathcal I\bigr\}
\ \cup\
\bigl\{\sigma\in G:\ \sigma(z)=z\ \text{for some }z\in\Omega\bigr\}
\]
with $\mathcal S(g)$ as in \eqref{ap:Sg}.  Having a fixed point in a
$G$-set is conjugation-invariant (if $\sigma(x)=x$, then
$(\tau\sigma\tau^{-1})(\tau x)=\tau x$), so $A_g$ is conjugacy-stable, as
the theorem asserts.

Fix $p\notin \mathcal S(g)$ and a prime $\mathfrak p$ of $\cO_L$ over $p$, and let
$\sigma=\Frob_{\mathfrak p}$, which is defined by
Lemma~\ref{lem:model}\ref{mod:frob}.  Since $g$ has integer coefficients
and $\F_p\subseteq\F_{\mathfrak p}$, the reduction of $g$ modulo
$\mathfrak p$ agrees on $(\F_p^\times)^2$ with the reduction of $g$
modulo $p$; hence
\[
p\in\PP(g)\Longleftrightarrow
\bigl(Z(\bar g)\cap(\overline{\F_p}^\times)^2\bigr)(\F_p)\ne\emptyset ,
\]
and by the reduced decomposition \eqref{ap:decomp-bar} of
Lemma~\ref{lem:model}\ref{mod:decomp} the right-hand side holds if and
only if
\[
\Bigl(\bigcup_{i\in\mathcal I}\bar U_i\Bigr)(\F_p)\ne\emptyset
\quad\text{or}\quad \bar\Omega(\F_p)\ne\emptyset ,
\]
which by Lemma~\ref{lem:count} is equivalent to $\sigma$ fixing an
element of $\mathcal I$ or of $\Omega$, i.e.\ to $\sigma\in A_g$.  This is
\eqref{eq:star}, completing the proof of
Theorem~\ref{thm:torus}.

\subsubsection*{Consequences}
Recall that a set of primes is \emph{Frobenian} in the sense of Serre
\cite[Subsection~3.3]{Serre} when it coincides, up to a finite set, with the set of
primes whose Frobenius class in some finite Galois group falls into a
prescribed conjugacy-stable subset.  

Theorem~\ref{thm:torus} says that, outside the finite set $\mathcal S(g)$,
$\PP(g)$ coincides with the set of primes whose Frobenius class in
$G=\Gal(L/\Q)$ lies in $A_g$.  Since $A_g$ is conjugacy-stable, the
Chebotarev density theorem \cite[Ch.~VII, (13.4)]{Neukirch} assigns to
this latter set the natural density $|A_g|/|G|$.  The two sets coincide
away from the finite set $\mathcal S(g)$, so they can differ only at the
finitely many primes of $\mathcal S(g)$; a finite change does not alter
the density, hence $\PP(g)$ has density $|A_g|/|G|$.  Its complement in
$\PP$, namely $\mathfrak D(H)$, is therefore Frobenian of natural density
$d(H)\in\Q\cap[0,1]$, given by
\[
d(H)=
\begin{cases}
1, & D=0,\\[2pt]
1-\dfrac{|A_g|}{|G|}, & D\ge1.
\end{cases}
\]
Since $\mathfrak D(H)=\PP\setminus\PP(g)$, its shape is governed by
$A_g$ alone, and the three possibilities for $A_g$ give the following
corollary.
\begin{cor}
\label{cor:trichotomy}
Let $H\le F$ be free of rank two with $A_H=F^{\ab}$.  Exactly one of the
following holds.
\begin{enumerate}[label=(\roman*), ref=(\roman*)]
\item\label{tri:all} $A_g=\emptyset$ and $I_H=\R$, and $\mathfrak D(H)=\PP$.
\item\label{tri:fin} $A_g=G$ and $I_H\ne\R$, and $\mathfrak D(H)$ is finite,
being contained in $\mathcal S(g)$.
\item\label{tri:pos} $A_g\ne G$ and $I_H\ne\R$, and $\mathfrak D(H)$ is
neither finite nor cofinite.
\end{enumerate}
\end{cor}
\begin{proof}
If $D=0$ then $g=1$, so $\PP(g)=\emptyset$, $I_H=\R$ and
$A_g=\emptyset$; by Theorem~\ref{thm:main} we have
$\mathfrak D(H)=\PP$, which is case \ref{tri:all}.

Now $D\ge1$.  Since $\mathcal I\ne\emptyset$ and the identity element of
$G$ fixes every element of every $G$-set, $1\in A_g$.  By
Theorem~\ref{thm:torus}, for every prime $p\notin\mathcal S(g)$ we have
$p\in\PP(g)$ if and only if $\Frob_{\mathfrak p}\in A_g$.  If $A_g=G$,
every such $p$ lies in $\PP(g)$, hence $\mathfrak D(H)\subseteq
\mathcal S(g)$ is finite by Theorem~\ref{thm:main}; this is case
\ref{tri:fin}.  If $A_g\ne G$, then $1\in A_g$ and $A_g\ne G$, so the
density formula above gives $0<d(H)=1-|A_g|/|G|<1$; hence
$\mathfrak D(H)$ is neither finite nor cofinite, which is case \ref{tri:pos}.
\end{proof}

\begin{cor}
\label{cor:decide}
Let $w_1,w_2$ be words in $x_1^{\pm1},x_2^{\pm1}$ such that
$H=\langle w_1,w_2\rangle$ is free of rank two with $A_H=F^{\ab}$.  There is an
algorithm which, given $w_1,w_2$,
\begin{enumerate}[label=(\alph*), ref=(\alph*)]
\item\label{dec:D} computes $\mathfrak D(H)$: it returns a normalized generator
$g$ of $I_H$, the data $L$, $G$, $A_g$ and $\mathcal S(g)$ of
Theorem~\ref{thm:torus}, and the verdicts at the primes of $\mathcal S(g)$,
which together determine $\mathfrak D(H)$ completely;
\item decides which of the cases \ref{tri:all},
\ref{tri:fin}, \ref{tri:pos} of Corollary~\ref{cor:trichotomy} holds, and in
case \ref{tri:fin} outputs the complete finite list of the elements of
$\mathfrak D(H)$;
\end{enumerate}
In particular, $d(H)$ is computable.
\end{cor}

\begin{proof}
By Proposition~\ref{prop:principal} and
Remark~\ref{rem:g-depends-on-basis} the element $g\in\R$ determined by
$\ab_{[F,F]}([w_1,w_2])=g\cdot\theta$ generates $I_H$.  It is computable:
Reidemeister--Schreier rewriting of the explicit word $[w_1,w_2]$
with respect to the transversal $\{x_1^ax_2^b\}$ of $[F,F]$ in $F$
\cite{MKS}, followed by abelianization, returns $g$.  By
Remark~\ref{rem:g-depends-on-basis} we may then replace $g$ by a normalized
generator, that is assume \eqref{eq:norm-g}.

 If $D=\deg g=0$ we are in case
\ref{tri:all}, then $\mathfrak D(H)=\PP$ by Corollary~\ref{cor:trichotomy}.  

If $D\ge1$, we compute the finite data
$(L,G,\mathcal I,\Omega,A_g,N,\mathcal S(g))$ of Theorem~\ref{thm:torus}.
By Remark~\ref{rem:omega-calc}, $L$, $\Omega$, and the action of $G$ on
$\mathcal I\cup\Omega$ are computable, so the finite set $A_g$, a union of
the two fixed-point conditions that define it, is computable; and
Remark~\ref{rem:Acompute} gives $N$ explicitly, so
$\mathcal S(g)=\{p\mid N\}\cup\{p\le9D^4\}$ is explicitly enumerated.

Comparing the two finite sets $A_g$ and $G$ now decides between case
\ref{tri:fin} ($A_g=G$) and case \ref{tri:pos} ($A_g\ne G$), and yields
$|A_g|$, so $d(H)=1-|A_g|/|G|$ is computable as the
Corollary asserts.  

For each prime $p\in\mathcal S(g)$ we decide $p\in\PP(g)$, hence
$p\in\mathfrak D(H)$ (Theorem~\ref{thm:main}), by running over the
$(p-1)^2$ points of $(\F_p^\times)^2$ and evaluating $g$ there; these are
the verdicts at the primes of $\mathcal S(g)$.  For $p\notin\mathcal S(g)$
Theorems~\ref{thm:main} and~\ref{thm:torus} give $p\in\mathfrak D(H)$ if
and only if $\Frob_{\mathfrak p}\notin A_g$.  Together these determine
$\mathfrak D(H)$ completely, which is \ref{dec:D}.  In case \ref{tri:fin}
Corollary~\ref{cor:trichotomy} gives $\mathfrak D(H)\subseteq
\mathcal S(g)$, so the verdicts output $\mathfrak D(H)$ in full.
\end{proof}

We close with an example in which the data $(L,G,\mathcal I,\Omega,A_g)$ of
Theorem~\ref{thm:torus} are computed in full.

\begin{ex}
Take
\[
g=(t_1^2-2t_1+t_2+1)^2+(t_2-1)^2 ,
\]
with $g(1,1)=1$, hence normalized, and with $D=4$.  Write
$a=t_1^2-2t_1+t_2+1$ and $b=t_2-1$, so that $g=a^2+b^2$.  Over
$\overline\Q$,
\[
g=h_1h_2,\qquad h_{1}=a+ i\,b,~h_{2}=a- i\,b
\]
two nonassociate absolutely irreducible polynomials of leading coefficient
$1$, so $\mathcal I=\{1,2\}$ and $c=1$.  Neither factor is defined over
$\Q$, and $L=\Q(i)$ with $G=\{1,\sigma\}$, $\sigma$ complex conjugation,
interchanging $h_1$ and $h_2$; hence $\sigma$ fixes no $i\in\mathcal I$.
The two curves $Z_1,Z_2$ meet where $a=b=0$, that is $t_2=1$ and
$t_1^2-2t_1+2=0$, so
\[
\Omega=\{(1+i,1),\ (1-i,1)\} ,
\]
a pair of torus points interchanged by $\sigma$ and fixed by neither.
Therefore $A_g=\{1\}$, the identity being the only element of $G$ fixing an
index or a point, so in particular $A_g\ne G$.  By
Proposition~\ref{prop:realize} a subgroup $H\le F$ with $A_H=F^{\ab}$
realizes this $g$, with $I_H=(g)$.  The formula displayed before
Corollary~\ref{cor:trichotomy} then gives
$d(H)=1-|A_g|/|G|=1-\tfrac12=\tfrac12$.
\end{ex}

\appendix

\section{Ruppert's criterion under reduction}
\label{app:ruppert}

This appendix supplies the input needed by
Lemma~\ref{lem:model}\ref{mod:fact}: absolute irreducibility of a
bivariate polynomial over a number field survives reduction at all but
finitely many primes.  The statement applies to any subring in which
finitely many prescribed elements are inverted.

\begin{prop}
\label{prop:red-irr}
Let $E$ be a number field and let $h\in E[t_1,t_2]$ be absolutely
irreducible of positive total degree, with leading coefficient $1$, divisible by neither $t_1$ nor
$t_2$.  One can compute a finite subset $\mathcal A(h)\subset E^\times$
with the following property.  Let $A\subseteq E$ be a subring in which
every nonzero coefficient of $h$ and every element of $\mathcal A(h)$ is a
unit, and let $\mathfrak q\subset A$ be a maximal ideal.  Then
$h\in A[t_1,t_2]$, the reduction $\bar h\in(A/\mathfrak q)[t_1,t_2]$ has
the same support as $h$ --- in particular the same bidegree, the same
total degree and the same leading coefficient --- and $\bar h$ is
absolutely irreducible.
\end{prop}

\emph{Ruppert's criterion in matrix form.} Ruppert \cite{Ruppert} detects absolute irreducibility of a bivariate
polynomial by linear algebra.  We use the matrix version of
\cite[Section~2]{Ruppert}, which behaves well under reduction.

Write $\deg f=(m_1,m_2)$ for $\deg_{t_1}f=m_1$, $\deg_{t_2}f=m_2$, and let
$\deg f\le(a,b)$ mean $\deg_{t_1}f\le a$, $\deg_{t_2}f\le b$; a negative
bound forces $f=0$.  Let $\mathcal{K}$ be any field and let
$f\in\mathcal{K}[t_1,t_2]$ have bidegree exactly $(m_1,m_2)$, with
$m_1,m_2\ge1$.  Look for pairs $(r,s)\in\mathcal{K}[t_1,t_2]^2$ with
$\deg r\le(m_1-1,m_2)$ and $\deg s\le(m_1,m_2-2)$ satisfying the Ruppert
equation
\begin{equation}
\partial_{t_2}(r/f)=\partial_{t_1}(s/f) .
\label{ap:ruppert}
\end{equation}
Let $w$ be the column vector of the $2m_1m_2+m_2-1$ unknown
coefficients of $r$ and $s$.  Since $f\ne0$, multiplying
\eqref{ap:ruppert} by $f^2$ is reversible and turns it into the identity
\[
f\,\partial_{t_2}r-r\,\partial_{t_2}f-f\,\partial_{t_1}s+s\,\partial_{t_1}f=0 ,
\]
Its coefficients are linear forms in $w$.  From the explicit formulas in
\cite[p.~65]{Ruppert} we get two facts.  These forms are indexed by
the monomials $t_1^kt_2^l$ with $(k,l)\le(2m_1-1,2m_2-2)$, giving
$2m_1(2m_2-1)$ equations.  Moreover each coefficient of each unknown is
an integer multiple of a single coefficient of $f$, of absolute value at most
$\max(m_1,m_2)$.  Thus \eqref{ap:ruppert} is a homogeneous linear system
\[
M(f)\,w=0 ,\qquad
M(f)\in\mathcal{K}^{\,2m_1(2m_2-1)\times(2m_1m_2+m_2-1)} ,
\]
The matrix $M(f)$ is determined by $f$ \emph{together with the declared
bidegree} $(m_1,m_2)$, which fixes the shape of the matrix.

Two consequences of this shape will be used repeatedly.  The matrix $M$
is $\Z$-linear in the coefficients of $f$: $M(\lambda f)=\lambda M(f)$ for
scalars $\lambda$, and $\varphi\bigl(M(f)\bigr)=M(f^{\varphi})$ entrywise
for a ring homomorphism $\varphi$ of coefficients, since
$\varphi(cx)=c\varphi(x)$ for $c\in\Z$.  Both sides use the same declared
bidegree $(m_1,m_2)$.  Moreover $\mathrm{rank}\,M(f)$ does not change when
$\mathcal{K}$ is enlarged: the rank is the largest $r$ admitting a
nonvanishing $r\times r$ minor, and a single element of $\mathcal{K}$ is
zero exactly when it is zero in any extension field of $\mathcal{K}$.

\begin{lem}[Ruppert]
\label{lem:ruppert}
Let $\mathcal{K}$ be a field and let $f\in\mathcal{K}[t_1,t_2]$ have
bidegree $(m_1,m_2)$ with $m_1,m_2\ge1$.
\begin{enumerate}[label=(\roman*), ref=(\roman*)]
\item\label{rup:free} If $\mathrm{rank}\,M(f)=2m_1m_2+m_2-1$, then $f$ is
absolutely irreducible.  No hypothesis is made on
$\mathrm{char}\,\mathcal{K}$.
\item\label{rup:zero} If $\mathrm{char}\,\mathcal{K}=0$ and $f$ is
absolutely irreducible, then $\mathrm{rank}\,M(f)=2m_1m_2+m_2-1$.
\end{enumerate}
\end{lem}

\begin{proof}
Both assertions are the two criteria in \cite[p.~65]{Ruppert},
stated there over an algebraically closed field.  They extend to any
$\mathcal{K}$.  Indeed $M(f)$ has the same rank over $\mathcal{K}$ as over
$\overline{\mathcal{K}}$, and $f$ keeps its bidegree, so $M(f)$ is the same
matrix in $\overline{\mathcal{K}}[t_1,t_2]$.  There, a nonzero solution
$(r,s)$ of \eqref{ap:ruppert} is just a nonzero kernel vector of $M(f)$.
We argue over $\overline{\mathcal{K}}$.  In that field a polynomial of
bidegree $(m_1,m_2)$ with $m_1,m_2\ge1$ is neither zero nor a unit, so
``reducible'' in \cite{Ruppert} means ``not absolutely irreducible''.

Lemma~1 of \cite{Ruppert}, valid in every characteristic, turns a
factorization of $f$ over $\overline{\mathcal{K}}$ into a nonzero solution
of \eqref{ap:ruppert} (its Case~II covers repeated factors), hence into a
rank deficiency of $M(f)$.  So full column rank makes $f$ irreducible in
$\overline{\mathcal{K}}[t_1,t_2]$, which is \ref{rup:free}.
Lemma~2 of \cite{Ruppert} says that over an algebraically closed field of
characteristic $0$, a nonzero solution of \eqref{ap:ruppert} makes $f$
reducible.  Its hypothesis $m_2\ge1$ is met.  Hence, over characteristic
$0$, an absolutely irreducible $f$ has no nonzero solution, so $M(f)$ has
full column rank; this is \ref{rup:zero}.
\end{proof}

\begin{proof}[Proof of Proposition~\ref{prop:red-irr}]
Suppose $h$ has bidegree $(m_1,m_2)$ with $m_1,m_2\ge1$.  Choose
$d\in\cO_E\setminus\{0\}$ with $dh\in\cO_E[t_1,t_2]$.  As $h$ is absolutely
irreducible and $\mathrm{char}\,E=0$, Lemma~\ref{lem:ruppert}\ref{rup:zero}
makes $M(h)$ of full column rank; so is $M(dh)=d\,M(h)$, a matrix over
$\cO_E$.  Pick a nonvanishing maximal minor $\Delta$ of $M(dh)$ and set
$\mathcal A(h)=\{d,\Delta\}$.

Now let $A$ and $\mathfrak q$ be as in the statement.  The coefficients of
$h$ lie in $A$, and the nonzero ones are units, so their reductions are
nonzero and $\bar h$ has the same support as $h$.  In particular the
declared bidegree $(m_1,m_2)$ is unchanged.  The entries of $M(dh)$ are
integer multiples of coefficients of $dh$, hence lie in $A$, and
$\Z$-linearity of $M$ gives $\overline{M(dh)}=M(\bar d\bar h)=\bar d\,M(\bar h)$
over $A/\mathfrak q$.  The left side keeps the nonvanishing minor
$\bar\Delta$, so $M(\bar h)$ has full column rank; indeed it is $\bar d^{-1}$
times the left side.  Lemma~\ref{lem:ruppert}\ref{rup:free} needs no
hypothesis on the characteristic, so $\bar h$ is absolutely irreducible.

If instead $h$ has degree $0$ in one variable, then $h$ is an irreducible
univariate polynomial over $\overline\Q$, hence linear.  Not being a
monomial, it is $h=t_1-\alpha$ or $h=t_2-\alpha$ with $\alpha\in E^\times$.
Set $\mathcal A(h)=\{\alpha\}$.  As above, $\bar\alpha\ne0$, so $\bar h$ has
the same support as $h$, and a linear polynomial is absolutely irreducible.

Both branches are constructive: the factorization deciding which branch
applies, the denominator $d$, and a nonvanishing maximal minor $\Delta$ of
the explicit matrix $M(dh)$ are all computable from the coefficients of $h$
\cite{GG}.
\end{proof}

\section{Proof of Lemma~\ref{lem:radical}}
\label{app:radical}

We prove the two equalities
$\mathcal{J}_{ij}=\bigcap_{z\in\Omega_{ij}}\mathfrak m_z=\sqrt{\mathfrak L_{ij}}$ of
Lemma~\ref{lem:radical}.

\begin{proof}[Proof of Lemma~\ref{lem:radical}]
Distinct $z,w\in\Omega_{ij}$ differ in some
coordinate, so $\mathfrak m_z+\mathfrak m_w=L[t_1,t_2,u]$; the ideals
$\mathfrak m_z$ are hence pairwise comaximal, and by the Chinese remainder
theorem their product equals their intersection.  Thus
$\mathcal J_{ij}=\bigcap_{z\in\Omega_{ij}}\mathfrak m_z$.

It remains to identify $\sqrt{\mathfrak L_{ij}}$.  Let
$\Lambda=L[t_1,t_2,u]/\mathfrak L_{ij}$.  The
$h_i,h_j$ are nonassociate irreducibles of the unique factorization domain
$\overline\Q[t_1,t_2]$, hence coprime; so $(h_i,h_j)$ is zero-dimensional
and $\Lambda_0=L[t_1,t_2]/(h_i,h_j)$ is a finite-dimensional $L$-algebra.  The relation $t_1t_2u=1$ makes
$\Lambda=\Lambda_0[1/t_1t_2]$ a localization of $\Lambda_0$, hence again
Noetherian of dimension $0$. In dimension $0$ every prime is maximal, so the nilradical $\sqrt{0_\Lambda}$ of $\Lambda$ is the intersection of the maximal ideals.

These maximal ideals are exactly the ones accounting for $\Omega_{ij}$.  If
$\mathfrak M\subset\Lambda$ is maximal, then $L'=\Lambda/\mathfrak M$ is a
finite extension of $L$.  Any $L$-embedding $L'\hookrightarrow\overline\Q$
sends the images of $t_1,t_2,u$ to a point $\zeta\in Z(\mathfrak L_{ij})$
with $L'=L(\zeta)$.  Every point of $Z(\mathfrak L_{ij})$ projects through
$\pr$ onto a point of $Z(h_i,h_j)\cap(\overline\Q^\times)^2=\Omega_{ij}$,
with $u$ then equal to $(t_1t_2)^{-1}$; as $\Omega_{ij}\subset(L^\times)^2$
by the construction of $L$, all coordinates of $\zeta$ lie in $L$.  Thus
$L'=L$, and $\mathfrak M$ is the image of $\mathfrak m_z$ for the
corresponding $z\in\Omega_{ij}$; conversely each $\mathfrak m_z$ contains
$\mathfrak L_{ij}$ and maps to the maximal ideal of $\Lambda$ at $z$.  Pulling
back the description of $\sqrt{0_\Lambda}$ along
$L[t_1,t_2,u]\to\Lambda$ therefore yields the second equality
$\sqrt{\mathfrak L_{ij}}=\bigcap_{z\in\Omega_{ij}}\mathfrak m_z$.

When $\Omega_{ij}=\emptyset$ the same identity holds with the empty
intersection: the analysis above leaves $\Lambda$ with no maximal ideal, so
$\Lambda=0$ and $\sqrt{\mathfrak L_{ij}}=(1)$; likewise $\mathcal J_{ij}=(1)$
by convention, so both sides of the equality are the unit ideal.

\end{proof}

\end{document}